\documentclass[11pt,letterpaper]{article}
\usepackage[utf8]{inputenc}
\usepackage[left=1in, right=1in, bottom = 0.9in, top = 0.9in]{geometry}
\usepackage[T1]{fontenc}
\usepackage[english]{babel}
\usepackage{amsmath}
\usepackage{amsfonts}
\usepackage{amssymb}
\usepackage{amsthm}
\usepackage{listings}
\usepackage{xcolor}
\usepackage{cite}
\usepackage{xypic}
\usepackage{pgf,tikz,pgfplots}
\pgfplotsset{compat=1.15}
\pgfplotsset{ticks = none}
\usepackage{mathrsfs}
\usetikzlibrary{arrows, decorations,calc}
\usepackage{verbatim}
\usepackage{mathtools}
\usepackage{multicol} 
\usepackage{hyperref}
\usepackage{abstract}
\usepackage{appendix}
\usepackage{graphicx}
\usepackage{caption}
\usepackage{subcaption}
\usepackage[shortlabels]{enumitem}

\usepackage{dsfont}

\newcommand*{\N}{\mathbb{N}}
\newcommand*{\Z}{\mathbb{Z}}

\newcommand*{\R}{\mathbb{R}}

\newcommand*{\II}{\mathcal{I}}
\newcommand*{\RR}{\mathcal{R}}
\newcommand*{\DD}{\mathcal{D}}
\newcommand*{\SG}{\mathsf{SG}}
\newcommand*{\uu}{\mathsf{u}}

\newcommand{\En}{\mathcal{E}}
\newcommand{\level}{\mathrm{level}}
\newcommand{\doubleTriangle}{{\triangle\!\triangle}}
\newcommand{\Dom}{\mathrm{Dom}}

\newcommand{\E}{\mathds{E}}
\newcommand{\domTriangle}{\mathsf{Tr}}
\newcommand{\domTriangleBoundary}{\partial\domTriangle}
\newcommand{\greenFunction}{\mathsf{G}}
\newcommand{\greenFunctionDiscrete}{\mathsf{g}}
\newcommand{\greenOperator}{\boldsymbol{G}}
\newcommand{\greenOperatorDiscrete}{\boldsymbol{g}}
\newcommand{\greenFunctionUnscaled}{\mathrm{g}}
\newcommand{\tridots}{{\mathrel{\raisebox{0pt}{.}}\!\mathrel{\raisebox{2pt}{.}}\!\mathrel{\raisebox{0pt}{.}}}}
\newcommand{\dirac}[2]{\mathds{1}_{\{#1\}}(#2)}

\renewcommand{\phi}{\varphi}
\renewcommand{\epsilon}{\varepsilon}
\renewcommand{\theta}{\vartheta}

\DeclareMathOperator{\supp}{supp}

\newtheorem{lem}{Lemma}[section]
\newtheorem{prop}{Proposition}[section]
\newtheorem{thm}{Theorem}[section]
\newtheorem{defn}{Definition}[section]
\newtheorem{corollary}{Corollary}[section]

\pgfdeclaredecoration{quasi-sirpinski}{do}{%
    \state{do}[width=\pgfdecoratedinputsegmentlength, next state=do]{%
        \pgfpathmoveto{\pgfpointpolar{-60}{\pgfdecoratedinputsegmentlength/2}}%
        \pgfpathlineto{\pgfpointorigin}%
        \pgfpathlineto{\pgfpoint{\pgfdecoratedinputsegmentlength/2}{0pt}}%
        \pgfpathclose%
    }
}
\usetikzlibrary{fadings}
\usetikzlibrary{patterns}

\renewenvironment{thebibliography}[1]{
  \begin{oldthebibliography}{#1}
    \setlength{\itemsep}{0.25em}
}
{
  \end{oldthebibliography}
}

\title{Internal aggregation models with multiple sources and obstacle problems on Sierpi\'nski gaskets}

\date{\today}
\author{Uta Freiberg, Nico Heizmann, Robin Kaiser, Ecaterina Sava-Huss}
\begin{document}
\maketitle
\begin{abstract}
We consider the doubly infinite Sierpi\'nski gasket graph $\mathsf{SG}_0$, rescale it by factor $2^{-n}$, and on the rescaled 
graphs $\mathsf{SG}_n=2^{-n}\mathsf{SG}_0$, for every $n\in\N$, we investigate the limit shape of three aggregation models with  initial configuration $\sigma_n$ of particles supported on multiple vertices.  The models under consideration are:
{\it divisible sandpile} in which the excess mass is distributed among the vertices until each vertex is stable and has mass less or equal to one, {\it internal DLA} in which particles do random walks until finding an empty site, and {\it rotor aggregation} in which particles perform deterministic counterparts of random walks until finding an empty site.
We denote by $\mathsf{SG}=cl\left(\cup_{n=0}^{\infty}\mathsf{SG}_n\right)$ the infinite Sierpi\'nski gasket, which is a closed subset of $\mathbb{R}^2$, for which $\mathsf{SG}_n$ represents the level-$n$ approximating graph, and we consider a continuous function $\sigma:\mathsf{SG}\to\mathbb{N}$. For $\sigma$ we solve the obstacle problem and we describe the noncoincidence set $D\subset \mathsf{SG}$ as the solution of a free boundary problem on the fractal $\mathsf{SG}$. If the discrete particle configurations $\sigma_n$ on the approximating graphs $\mathsf{SG}_n$ converge pointwise to the continuous function $\sigma$ on the limit set $\mathsf{SG}$, we prove that, as $n\to\infty$, the scaling limits of the three aforementioned models on $\mathsf{SG}_n$ starting with initial particle configuration $\sigma_n$ converge to the deterministic solution $D$ of the free boundary problem on the limit set $\mathsf{SG}\subset\mathbb{R}^2$. For $D$ we also investigate boundary regularity properties.
\end{abstract}
\vspace{-0.3cm}
\section{Introduction}

{\it Internal DLA} (shortly IDLA) on $\Z^d$, as a special case of the Diaconis-Fulton smash sum\cite{diaconis_general_idla}, was introduced in \cite{idla_lawler}, and it is a random cluster growth model in which $m$ particles are initially located at the origin $o\in\Z^d$, and each of them performs a simple random walk until arriving at a site that has not been visited before, at which site the particle stops and occupies that position forever. The resulting random subset of occupied sites in $\Z^d$ is called \textit{internal DLA} cluster and has, for $m\to\infty$, a deterministic limit shape as shown in \cite{idla_lawler}. 
{\it Divisible sandpile} is a deterministic growth model introduced in \cite{spherical_asymptotics_rotor} as a potential theoretical tool to approach internal DLA, and relaxes the integrability condition in the similar model, {\it the Abelian sandpile}. Divisible sandpile uses a continuous amount of mass, and any site with mass more than one is unstable and topples by keeping mass one for itself and distributing the rest equally among the neighbors. As proved in \cite{spherical_asymptotics_rotor}, as time goes to infinity, the sequence of mass configurations converges pointwise to a limit configuration in which each site has mass $\leq 1$. By starting with mass $m$ in $o\in\Z^d$ and zero everywhere else, the subset  of toppled sites in $\Z^d$ in the limit configuration, called {\it the divisible sandpile cluster}, has as limit shape an Euclidean ball, like in the internal DLA case. Finally, in the deterministic counterpart of internal DLA, called {\it rotor aggregation} \cite{rotor-aggreg-lev-peres}, $m$ particles starting at the origin $o\in\mathbb{Z}^d$ perform rotor walks (deterministic analogues of random walks) until finding sites unvisited previously, where they settle forever. The set of occupied sites is called {\it rotor cluster} and has on $ \Z^d$ the same limit shape as internal DLA and divisible sandpile as shown in \cite{spherical_asymptotics_rotor}. All  three  models are Abelian, in the sense that, for the resulting final configuration or  for the set of occupied sites, it does not matter in which order particles move or in which order sites are being stabilized.

For single point sources of $m$ particles starting at a fixed vertex $o$ in the underlying state space, there has been intensive work that concerns mainly the limit shape and the fluctuations of the cluster around the limit shape. While it is believed that the three models share the same limit shape on any state space, so far the available limit results rely very much on the geometry and the growth of those spaces, and in particular on the long term behavior of random and rotor walks on them. Other than $\Z^d$, we mention here a few of the state spaces where the limit shapes for the above models and particles starting at the same vertex have been investigated: on {\it trees} internal DLA \cite{idla-trees} and rotor aggregation \cite{rr-trees-lionel}; on {\it comb lattices} \cite{wilfried-eca-comb, rr-comb}; internal DLA  on {\it cylinder graphs} in \cite{idla-cylinders}; on {\it Sierpi\'nski gasket graphs} internal DLA in \cite{idla_cati},  divisible sandpile in \cite{Div-Sandpile-SG}, and  rotor aggregation together with Abelian sandpile in \cite{limit-shape-rotor-div-sandpile}. In particular, the  Sierpi\'nski gasket graph is the only non-trivial state space (other than $\Z$) where even a fourth model, {\it the Abelian sandpile} shares the same limit shape with the other three above introduced models. We do not focus on this model, but we would like to emphasize that on $\Z^d$ the scaling limit for the Abelian sandpile exists but a precise description is still an open problem.  

While the  above results focus on initial mass or particle configurations supported on a single vertex, it is natural to extend these configurations to general functions with bounded support, and to describe the limit shapes and fluctuations. A first step in this direction has been done in \cite{LP-mult-idla}, where the authors investigate scaling limits for the aggregation models with multiple sources on lattices $\delta_n\Z^d$, with $\delta_n\downarrow 0$, so the processes are run on finer and finer lattices. Given a function $\sigma$ on $\R^d$, there is a natural way to construct a corresponding obstacle $\gamma$ for $\sigma$ and to investigate its smallest superharmonic majorant $s$. Then the {\it noncoincidence set for the obstacle problem} $D$ is the set $D:=\{x\in\R^d:\ s(x)>\gamma (x)\}$; for more details on obstacle problems and free boundary problems in $\R^d$ we refer the reader to \cite{free-pde-book}. In \cite{LP-mult-idla}, the authors consider the three models of cluster growth on  lattices $\delta_n\Z^d$ with starting initial configurations $\sigma_n$ on $\delta_n \Z^d$ converging in an appropriate way to $\sigma $ on $\R^d$ and they prove that the resulting sets of occupied sites converge in the Hausdorff metric to the noncoincidence set $D\subset\mathbb{R}^d$, for reasonable assumptions on the function $\sigma$. Their approach is based on potential theoretical analysis in $\R^d$ where most of the needed tools and estimates are well understood.

\textbf{Our contribution.} While it is not clear how to extend such an analysis on any arbitrary state space and how to pass from discrete potential theory to its continuous counterpart on some limiting object similar to $\R^d$ in order to get a precise understanding of the scaling limits for the aggregation models, the focus of this work is on a particular class of fractals: {\it Sierpi\'nski gaskets} $\mathsf{SG}$ as closed subsets of $\mathbb{R}^2$. On the discrete graph approximations $\mathsf{SG}_n$ of $\mathsf{SG}$, we have a good understanding of the random walk and its limiting process Brownian motion, and they are a rich source of objects with scaling-invariance and self-similar properties on which potential theory is well understood.
The limiting object $\mathsf{SG}$ which takes the role of $\mathbb{R}^d$ from \cite{LP-mult-idla}, is from the potential analytical point of view also well understood, even if the analysis on fractals is way different than on $\mathbb{R}^d$.  
In order to explore the scaling invariance of the graphs $\mathsf{SG}_n$ and the decimation invariance of random walks on them, unlike the case of $\mathbb{Z}^d$, we cannot rescale the graphs by any arbitrary sequence $\delta_n$ converging to $0$, and we take $\delta_n=2^{-n}$; the choice of the scaling constant will be made clear below. The main goal of the underlying work is to study obstacle problems for functions $\sigma$ on the limit gasket $\mathsf{SG}$ and to describe the noncoincidence set $D\subset \mathsf{SG}$ and its main properties, by a careful analysis on fractals. On the approximating graphs $\mathsf{SG}_n$ with discrete particle configurations $\sigma_n$ that converge to $\sigma$ on $\SG$ in a way to be described below, we show that the divisible sandpile cluster, the internal DLA cluster, and the rotor aggregation cluster share the same scaling limit that converges to the noncoincidence set $D\subset  \mathsf{SG}$, similar to $\Z^d$ and $\R^d$ from \cite{LP-mult-idla}. We would like to emphasize that the analysis from $\R^d$ does not immediately carry over to fractals, and there are several differences and difficulties to deal with on $\mathsf{SG}$; for instance on the Sierpi\'nski gasket, there is no known globally defined Green function and this makes the investigation of obstacle problems more subtle. One can overcome this issue on the gasket by defining the solution of the obstacle problem as a limit of solutions of obstacle problems on approximating graphs and showing that this is indeed well-defined.

\textbf{Main results.} 
 In order to state our  results, we first refer the reader to Section \ref{sec:sierp-gas}, where the level-$n$ approximations $\mathsf{SG}_n$ (infinite graphs) as rescalings of the doubly infinite  Sierpi\'nski gasket graph $\mathsf{SG}_0$ by the factor $2^{-n}$ and the ''continuous limiting object'' $\mathsf{SG}$ as a closed subset of $\R^2$, are  introduced.
Let $B(0,2^l)$ be the closed Euclidean ball of radius $2^l$  around $0$ in $\SG$, i.e. two compact triangles with side length $2^l$ joined together at $0$, where $l\in\N$, and set $\domTriangle_l=B(0,2^l)$. If we define the Green function  $\greenFunction_{\domTriangle_l} $ on $\SG\cap \domTriangle_l$ as in \eqref{eq:green-fc-sg}, and for a function $f:\domTriangle_l\to  \R$  we set
\begin{align*}
\greenOperator_{\domTriangle_l}f(x)=\int_{\domTriangle_l}\greenFunction_{\domTriangle_l}(x,y)f(y)d\mu(y),
\end{align*}

where $\mu$ is a multiple of the Hausdorff measure on $\domTriangle_l$, then we first prove the following. 
\begin{prop}\label{prop:obst1}
Let $\sigma:\mathsf{SG}\rightarrow[0,\infty)$ be bounded and continuous almost everywhere, such that $\supp(\sigma)$ is also bounded. We define the obstacles on $\domTriangle_l\subset\SG$ by 
\vspace{-0.2cm}
\begin{align*}
    \gamma_{\domTriangle_l}:=-\greenOperator_{\domTriangle_l}(\sigma - 1),
\end{align*}
and the superharmonic majorant of $\gamma_{\domTriangle_l}$ by
\begin{align*}
    s_{\domTriangle_l}:=\inf\{f(x):\ f\text{ is continuous, superharmonic and }f\geq \gamma_{\domTriangle_l}\}.
\end{align*}
Then for the functions $\uu_{\domTriangle_l}=s_{\domTriangle_l}-\gamma_{\domTriangle_l}$, the pointwise limit $u=\lim_{l\rightarrow\infty}\uu_{\domTriangle_l}$ exists and is well-defined. Furthermore, the set $D=\{u>0\}\subset \mathsf{SG}$ is bounded.
\end{prop}
Starting with an almost everywhere continuous function $\sigma$ on $\mathsf{SG}$ as in Proposition \ref{prop:obst1}, we consider discrete initial particle densities $\sigma_n$ on $\SG_n$ that converge to $\sigma$, and for $\sigma_n$ we perform divisible sandpile, internal DLA, and rotor aggregation on $\mathsf{SG}_n$. We denote the respective clusters by $\DD_n$, $\II_n$, and $\RR_n$, where the subscript $n$ indicates that we are on the approximating graph $\mathsf{SG}_n$ (i.e. we perform divisible sandpile, IDLA and rotor router aggregation with initial density $\sigma_n$ on $\mathsf{SG}_n$).
The next result shows that the scaling limit of these three models converges in the Hausdorff metric to the noncoincidence set $\tilde{D} = D \cup \{ \sigma \geq 1 \}^\circ$, where $D$ ist the set from Proposition \ref{prop:obst1}. 
We say that the sequence $(A_n)_{n\in\mathbb{N}}$ of sets $A_n\subseteq \mathsf{SG}_n$ converges to the set $A\subseteq \mathsf{SG}$, if for every $\varepsilon>0$, there exists $n_0\in \N$ such that for all $n\geq n_0$ it holds 
$$ A_{\varepsilon} \cap \SG_n \subseteq A_n\subseteq A^{\varepsilon},$$
where $A_{\varepsilon}$ and $A^\varepsilon$
are the inner and outer $\epsilon$-neighborhoods of $A$, respectively. A special case of the densities $\sigma_n$ is as in the result below, where $\sigma_n(x)$ can be obtained by averaging $\sigma$ over  a neighborhood of $x$ in $\SG$.
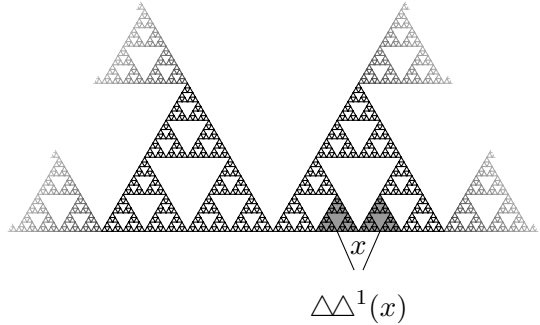
\begin{figure}
    \centering
    \pgfmathsetmacro{\siz}{3.5}
\pgfmathsetmacro{\fad}{1}

\begin{tikzpicture}[scale=0.65, decoration=quasi-sirpinski] 
      \draw decorate{ decorate { decorate { decorate{ decorate{ 
        (0,0) -- ++(60:\siz) -- ++(-60:\siz) -- cycle 
        } } } } };
    \draw decorate{ decorate { decorate { decorate{ decorate{ 
        (-\siz,0) -- ++(60:\siz) -- ++(-60:\siz) -- cycle 
        } } } } };
        	
        	\node (c) at (\siz/2,0) {};
        	\node[below] at (c) {$x$};
    	\fill[opacity = 0.4] (\siz/2,0) -- ++(60:\siz/4) -- ++(-60:\siz/4) -- cycle ;
    	\fill[opacity = 0.4] (\siz/2-\siz/4,0) -- ++(60:\siz/4) -- ++(-60:\siz/4) -- cycle ;
    	
    	\node (t) at (\siz/2,-1) {};
    	\node[below, node distance = 1pt] at (t) {$\doubleTriangle^1(x)$};
    	\draw (\siz/2-\siz/8,0) -- (t) -- (\siz/2+\siz/8,0);
    	
    \begin{scope}[shift = {(\siz/2,\siz/2*1.7320508076)}]
	\path[scope fading = east, fading angle = 27] (-\fad,-\fad) -- ++(\siz/2,\siz/2*1.7320508076) -- ++ (\siz/2,-\siz/2*1.7320508076);
			\draw decorate{ decorate { decorate { decorate{ decorate{ 
        (0,0) -- ++(60:\siz) -- ++(-60:\siz) -- cycle 
        } } } } };
	
	\end{scope}
	\begin{scope}[shift = {(\siz,0*1.7320508076)}]
	\path[scope fading = east, fading angle = 27] (-\fad,-\fad) -- ++(\siz/2,\siz/2*1.7320508076) -- ++ (\siz/2,-\siz/2*1.7320508076);
			\draw decorate{ decorate { decorate { decorate{ decorate{ 
        (0,0) -- ++(60:\siz) -- ++(-60:\siz) -- cycle 
        } } } } };
	
	\end{scope}
	\begin{scope}[shift = {(-2*\siz,0*1.7320508076)}]
	\path[scope fading = west, fading angle = -27] (\fad,-\fad) -- ++(\siz/2,\siz/2*1.7320508076) -- ++ (\siz/2,-\siz/2*1.7320508076);
			\draw decorate{ decorate { decorate { decorate{ decorate{ 
        (0,0) -- ++(60:\siz) -- ++(-60:\siz) -- cycle 
        } } } } };
	
	\end{scope}
	\begin{scope}[shift = {(-1.5*\siz,\siz/2*1.7320508076)}]
	\path[scope fading = west, fading angle = -27] (\fad,-\fad) -- ++(\siz/2,\siz/2*1.7320508076) -- ++ (\siz/2,-\siz/2*1.7320508076);
			\draw decorate{ decorate { decorate { decorate{ decorate{ 
        (0,0) -- ++(60:\siz) -- ++(-60:\siz) -- cycle 
        } } } } };
	
	\end{scope}
\end{tikzpicture}
    \caption{The double sided infinite Sierpi\'nski gasket and $\doubleTriangle^1(x)$ for $x=(1,0)$.}
    \label{fig:gasket+doubleTriangle}
\end{figure}
\begin{thm}\label{thm:uniform_scaling_limit}
Let $\sigma: \mathsf{SG}\rightarrow\N_0$ be bounded with compact support and continuous almost everywhere, such that $\{\sigma\geq 1\}=\overline{\{\sigma\geq 1\}^\circ}$. Let $\DD_n$, $\RR_n$ and $\II_n$ be the final set of occupied sites (resp. toppled sites in the sandpile) formed by divisible sandpile, rotor  aggregation and internal DLA respectively,  started from initial density
\begin{align}\label{eq:sigma-n}
    \sigma_n(x)=\Big\lfloor \mu(\doubleTriangle^{n+1}(x))^{-1}\int_{\doubleTriangle^{n+1}(x)}\sigma(y)d\mu(y)\Big\rfloor
\end{align}
on $\mathsf{SG}_n$, for every $n\in \N$. Then $\DD_n,\RR_n\to \widetilde{D}$ and $\II_n\to \widetilde{D}$ with probability one, as $n\to\infty$.
\end{thm}
For the definition of $\doubleTriangle^{n+1}(x)$ see equation \eqref{eq:triang-double}. Theorem \ref{thm:uniform_scaling_limit} is a summary of Theorem \ref{theo:convergenceSandpileDomains} (divisible sandpile scaling limit), Theorem \ref{thm:conv_domain_rotor} (rotor aggregation scaling) and Theorem 
\ref{thm:IDLA} (internal DLA scaling limit), for the choice \eqref{eq:sigma-n} of the discrete particle densities $\sigma_n$.

\textbf{Structure of the paper.} In Section \ref{sec:prelim}  we introduce the three aggregation models (divisible sandpile, internal DLA, rotor aggregation), the Sierpi\'nski gasket $\mathsf{SG}$ as well as its discrete graph approximations $\mathsf{SG}_n$ for $n\in\N$, and the potential theoretical tools on fractals. 
In Section \ref{sec:div-sand}  we prove the limit shape result for the divisible sandpile model, by defining the limiting process on $\mathsf{SG}$ as a limit of processes on compact subsets of $\mathsf{SG}$. 
Section \ref{sec:rotor} investigates the limit shape for rotor aggregation while Section \ref{sec:idla} deals with the internal DLA model.
Finally, in Appendices \ref{sec:appendixB} and \ref{sec:appendixC} we show how to calculate the solution of the free boundary problem for a specific starting configuration $\sigma$. We also show here that the boundary of the non-coincidence set of the free boundary problem is a null set with respect to the Hausdorff measure on $\mathsf{SG}$, as is needed in the proof of the scaling limit for  internal DLA.

\section{Preliminaries}\label{sec:prelim}

\subsection{Internal aggregation models}
Let $G=(V,E)$ be an infinite, locally finite and regular graph, i.e.~every vertex has degree  $\mathsf{d}\in\mathbb{N}$. If $x$ and $y$ are neighbors in $G$ we write $x\sim y$, and we denote by $d(x,y)$ the graph distance in $G$.

\begin{figure}
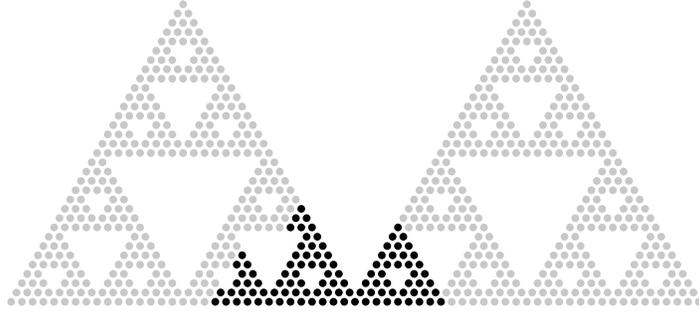

     \centering
     \begin{subfigure}[b]{0.8\textwidth}
         \centering
          \include{tikz_div}~\\[-1cm]
     \end{subfigure}
        \caption{Divisible sandpile for initial configuration  $\sigma=\mathds{1}_{B((-2,0),1)}+\mathds{1}_{B((-1/2,0),2)}$ on $\mathsf{SG}_2$.}
        \label{fig:aggregation_models_sim}
\end{figure} 
\textbf{Divisible sandpile.} A function $\nu_0:V\rightarrow\mathbb{R}$ is a sand distribution on $G$ if $\nu_0$ is non-negative and has finite support. The toppling operator at $x$ is
\begin{align*}
    T_x \nu_0=\nu_0+\max\{\nu_0(x)-1,0\}\Delta_G \mathds{1}_{\{x\}},
\end{align*}
where $\mathds{1}_{\{x\}}(\cdot)$ is the function on $V$ taking value $0$ everywhere except a single $1$ at the vertex $x$, and $\Delta_G$ is {\it the graph Laplacian of $G$}, i.e. the operator defined on real functions on $V$ given by
\begin{align*}
    \left(\Delta_G f\right) (x)= \frac{1}{\mathsf{d}}\sum_{z\sim x} f(z) - f(x).
\end{align*}
The toppling operator $T_x$ takes the mass exceeding $1$ at $x$ and distributes it equally among its neighbors.
For a sequence of vertices $(x_k)_{k\in\N}$,  define the  associated  toppling process by
$$\nu_n=T_{x_n}...T_{x_1}\nu_0,$$
where $\nu_n$ is the sand distribution after the first $n$ topplings have been performed. We remark that the total amount of mass is conserved during the toppling procedure. An important tool that will be used in the sequel is the  {\it odometer function}. For an initial sand distribution $\nu_0$ and a toppling sequence $(x_k)_{k\in\N}$ we define the $k$-th odometer function by
\begin{align*}
    u_k(y)=\sum_{j\leq k;x_j=y}\big(\nu_{j}(y)-\nu_{j-1}(y)\big),
\end{align*}
which keeps track of how much mass has been sent out of each vertex. If we topple along any sequence of vertices that contains each vertex in $V$ infinitely many times, then the sand distributions $(\nu_k)_{k\in\N}$ converge pointwise to some limit distribution $\nu$ with $\nu\leq 1$ on $G$. Also the sequence of odometer functions $(u_k)_{k\in\N}$ converges to a limit function $u$ called the odometer function of the divisible sandpile started from initial configuration $\nu_0$. It holds that $\nu= \nu_0 + \Delta_G u$ and the so-called {\it least action principle} states that $u$ is actually the minimal non-negative function that satisfies 
$$\nu_0 + \Delta_G u \leq 1.$$
See \cite{spherical_asymptotics_rotor}  for proofs of the pointwise convergence results for the odometer functions and mass distributions, and \cite[Lemma 3.2]{LP-mult-idla} for the least action principle. 
The {\it divisible sandpile cluster} is defined as the set of toppled sites $\mathcal{D}:=\{x\in G:u(x)>0\}$. In the literature, it is also common to define the divisible sandpile cluster as the set 
$\DD':= \{ x \in G : \nu(x) = 1 \}$ of fully occupied sites. The sets $\mathcal{D}$ and $\mathcal{D'}$ are related in that $u(x)>0$ implies $\nu(x)=1$ so $\mathcal{D}\subseteq\mathcal{D}'$, and $\nu(x)=1$ implies $u(x)>0$ or $\nu_0(x)=1$ or there exists a neighboring vertex $y\sim x$ with $u(y)>0$. Throughout this paper we will use the definition $\DD$ of the sandpile cluster, and the scaling limit of $\DD'$ follows immediately from that of $\DD$ by the above mentioned relation between the two sets.

\textbf{Internal DLA.}
For the infinite graph $G=(V,E)$, an initial configuration of particles is a function $\nu_0:V\rightarrow\mathbb{N}$ that is bounded and finitely supported. Denote by $m:=\sum_{v\in V}\nu_0(v)$ the total number of particles, and label the particles by integers $i=1,...,m$. Denote by $x_i$ the initial location of the particle labeled $i$, so  $ \#\{i|x_i=v\}=\nu_0(v)$ represents the number of particles located at $v$. For each $i=1,...,m$, let $\big(X_t(i)\big)_{t\in\N}$ be a simple random walk on $G$ starting at position $x_i$, and we construct inductively the internal DLA cluster  as follows. We set the cluster at time $0$ to be the empty set, and for $i\geq 1$ if the cluster at time $i-1$ has already been built, let the particle labeled $i$ start a random walk at its initial position $x_i$ and stop it after first exiting the previous cluster at time $i-1$. So the cluster at time $i$ is the union between the cluster at time $i-1$ and $\{X_{\tau_i}(i)\}$, where $\tau_i$ denotes the first exit time of the random walk $(X_t(i))_{t\in\mathbb{N}}$ from cluster at time $i-1$.
The stopping time $\tau_i$ represents the number of steps taken by particle labeled $i$, which is initially located at $x_i\in G$, until stopping at random location $X_{\tau_i}(i)$. Proceeding in this way for all $i=1,...,m$, we obtain the random subset of occupied sites in $G$ that we call {\it internal DLA} cluster. 
In the literature sometimes the notation $\II_n$ is used to denote the internal DLA cluster after $n$ particles settled, but throughout this paper, this notation is reserved for internal DLA cluster on the discrete graphs $\SG_n$ with particle configuration $\sigma_n$. 

\textbf{Rotor aggregation} (or rotor-router aggregation) is a deterministic counterpart of internal DLA and it is based on particles performing deterministic walks (rotor walks) instead of random walks. Rotor walks, also called rotor-router walks, were first studied by Priezzhev et al. in \cite{Priezzhev_1996} under the name {\it Eulerian Walkers}, and can be defined by means of a rotor configuration and a fixed cyclic ordering of the neighbors for each vertex. Denote the set of neighbors of $v\in V$ by $A_v$, and fix for any $v\in V$ a cyclic ordering $\mathsf{cyc}(v)$ of the neighbors of $v$. A rotor configuration is a function $\rho:V\rightarrow V$ such that for all $v\in V$,  $\rho(v)\in A_v$, so one can think 
of $\rho(v)$ as an arrow or a rotor that points to one of the neighbors of $v$. 
A particle starting at $v$ performs a {\it rotor walk} as follows: first it turns the rotor at $v$  from $\rho(v)$ to point to the next neighbor in the cyclic ordering $\mathsf{cyc}(v)$ of the neighbors of $v$, and then the particle follows this new direction. 

For rotor aggregation, we start with the same initial particle configuration $\nu_0:V\to\N$ which is bounded and finitely supported  as in the internal DLA. Denoting again by $m$ the total number of available particles and labeling the particles by $i=1,\ldots, m$, and also denoting the starting location of the particle labeled $i$ by $x_i$, we define rotor aggregation recursively exactly in the same way as internal DLA, with the difference that particles perform rotor walks instead of random walks until finding sites unvisited previously, where they stop. After all $m$ particles have found new sites, we obtain a deterministic subset of occupied sites that we call {\it the rotor cluster}. While in internal DLA particles perform independent random walks, in rotor aggregation the particle movements are deterministic but there is interaction among them. If one particle passed through some vertex until finding its final position, on the way there it modified the direction of several rotors and subsequent particles passing through the same vertex will use then the rotor direction left behind by previous visits. If $(Y_t{(i)})_{t\in \N}$ is a rotor walk starting from position $x_i$, then the rotor cluster with $m$ particles can be defined as follows: let the rotor cluster at time $0$ to be the empty set, and for $i\geq 1$ if the cluster at time $i-1$ has already been built, let the particle labeled $i$ start a rotor walk  with rotor configuration left behind by the previous particles, at its initial position $x_i$ and stop it after first exiting the previous rotor cluster at time $i-1$. If $\tau_i$ denotes the first exit time of the rotor walk $(Y_t{(i)})_{t\in\N}$ from the previous rotor cluster, then the rotor cluster at time $i$ is the union between the rotor cluster at time $i-1$ with $\{Y_{\tau_i}{(i)}\}$.
So $\tau_i$ represents the number of rotor steps taken by particle labeled $i$ initially located at $x_i\in G$, which stops at deterministic location $Y_{\tau_i}(i)$.

\subsection{Sierpi\'nski gasket}\label{sec:sierp-gas}

We recall the construction of the Sierpi\'nski gasket as in \cite{barlow-perkins-BM}.
Let $a_0=(0,0)$, $a_1=(1,0)$ and $a_2=(1/2,\sqrt{3}/2)$ be the corners of an equilateral triangle with side length $1$ and let $F_0=\{a_0,a_1,a_2\}$. Further let $J_0$ be the closed convex equilateral triangle with vertices $F_0$.  Inductively, for $n\geq 1$, we define
\begin{align*}
    F_{n+1}:=F_n\cup(2^n a_1+F_n)\cup(2^n a_2+F_n),
\end{align*}
so $F_n\subseteq 2^n J_0.$  Let $\mathsf{SG}_+=\bigcup_{n=0}^\infty F_n$
and denote by $\mathsf{SG}_0$ the union of $\mathsf{SG}_+$ with its reflection along the y-axis $\mathsf{SG}_-$. For $n\in\Z$, let
\begin{align*}
\delta_n:=2^{-n}, \quad    \mathsf{SG}_n:=\delta_n \mathsf{SG}_0, \quad \mathsf{SG}_\infty:=\bigcup_{n=0}^\infty \mathsf{SG}_n, \quad \mathsf{SG}_{-\infty}=\{0\}, \quad \mathsf{SG}:=\text{cl}(\mathsf{SG}_\infty).
\end{align*}
$\mathsf{SG}$ is a closed connected subset of $\mathbb{R}^2$ called {\it the double sided infinite Sierpi\'nski gasket} and the graphs $\mathsf{SG}_n$ represent the level-$n$ approximating graphs of $\mathsf{SG}$, where $x,y\in\SG_n$ are connected by an edge if and only if $|x-y|=\delta_n$ and the line connecting $x$ to $y$ is contained in $\mathsf{SG}$. For $x\in \mathsf{SG}_\infty$, we define its level as 
$$\level(x) := \inf\{n\in \Z:\ x \in \mathsf{SG}_n\}.$$ Clearly $0$ is the only vertex with $\level(0) = -\infty$. 
We recall below several properties of $\mathsf{SG}$ and the sequence of approximating graphs $(\mathsf{SG}_n)_{n\in \N}$. The {\it Hausdorff dimension $\alpha$} and the {\it walk dimension} $\beta$ of $\SG$ are given by
\begin{align*}
\alpha =\frac{\log3}{\log2}\approx 1.58496 \quad \text{and} \quad 
\beta =\frac{\log5}{\log2}\approx 2.32193.
\end{align*}
\textbf{Laplacian on $\SG_n$.} Recalling that by $\Delta_{\SG_n}$ we have denoted the graph Laplacian of the infinite graph $\SG_n$, we define 
the Laplacian on $\mathsf{SG}_n$ as a scaled version of $\Delta_{\mathsf{SG}_n}$ by
\begin{align*}
\Delta_n:=\delta_n^{-\beta}\Delta_{\mathsf{SG}_n}.
\end{align*}
We remark that by our choice of $\delta_n=2^{-n}$, we have $\delta_n^{-\beta}=5^n$, which is the scaling used in \cite{strichartz-differential-equations-fractals}. We call a function $h:\SG_n\to\R$ \textit{harmonic} if $\Delta_n h=0$, \textit{superharmonic} if $\Delta_n h\leq 0$ and \textit{subharmonic} if $\Delta_n h\geq 0$. For a function $\gamma:\SG_n\to\R$ its \textit{least superharmonic majorant} $s:\SG_n\to\R$ is defined as
$$s(x):=\inf\{f(x)|\ f:\SG_n\to\R \text{ is superharmonic and } f\geq \gamma\}.$$
For $n\in \N$,  denote by $B^n(x,r):=\{y\in\SG_n:\ d(x,y)\leq r\}$ the ball of radius $r\geq 0$ around $x\in\SG_n$ in the graph metric $d(\cdot,\cdot)$ of $\SG_n$. If $(X_t^n)_{t\in\N}$ is a simple random walk on $\SG_n$, and $B\subset \SG_n$, denote by $\tau_B^n(x)$ its first exit time from $B$:
\begin{align}\label{definition:firFprovidst-exit-time}
\tau_B^n:=\inf\{t\geq 0:\ X_t^n\notin B\}.
\end{align}
For simplicity of notation, if $B=B^n(x,r)$, then we write $\tau_r^n(x)$ instead of $\tau^n_{B^n(x,r)}$.
We will use for the rest absolute constants $C,c>0$ that might change from line to line.

\textbf{Properties of $\SG_n$, $n\in \N$.} We gather here several important features of $\SG_n$.
\vspace{-0.3cm}
\begin{enumerate}[(i)]
\setlength\itemsep{0em}
\item {\bf [Volume growth]} $B^n(x,r)$ has growth of order $\alpha$:
there exists $C>0$ such that
\begin{align}
C^{-1}r^\alpha\leq|B^n(x,r)|\leq Cr^\alpha. \tag{$V_\alpha$} \label{property:volume_growth}
\end{align}
See \cite{diffusion_fractals} and \cite{heat_kernels_and_sets_with_fractal_structures} for a proof and more details.
\item  {\bf [Elliptic Harnack inequality]} There exists $C>0$ such that for all $x\in \SG_n$, all $r>0$, and all functions $h>0$ with $\Delta_nh=0$ on $B^n(x,2r)$ we have
\begin{align}
    \sup_{y\in B^n(x,r)}h(y)\leq C\inf_{y\in B^n(x,r)}h(y) \tag{$EHI$}.\label{property:EHI}
\end{align}
See \cite{aof_kigami} for details.
\item {\bf [Expected exit time]} For every $x\in \SG_n$ and $r\geq 0$, there exists $C>0$ such that
\begin{align}
C^{-1}r^\beta\leq\mathbb{E}_y\big[\tau^n_r(x)\big]\leq Cr^\beta,\tag{$E_\beta$}\label{property:exit_time}
\end{align}
for the random walk $(X^n_t)_{t\in\mathbb{N}}$ starting at $y\in B^n(x,r)$; see \cite[Corollary 2.3a]{heat_kernels_and_sets_with_fractal_structures} for details.
\end{enumerate}

\textbf{Green functions and operators on $\SG_n$.} For any finite subset $\mathsf{T}\subseteq \mathsf{SG}_n$,  we denote its inner boundary by $\partial \mathsf{T} :=\{ v\in \mathsf{T}:\ \exists z \in \mathsf{T}^C \text{ with } z\sim v\} $. Consider again the simple random walk $(X_t^n)_{t\in \N}$ on $\SG_n$, and for $x,y\in \mathsf{T}$, we define the rescaled stopped Green function as 
\begin{align}\label{eq:green-visits}
    \greenFunctionDiscrete^n_{\mathsf{T}}(x,y)\!:= \!\Big(\frac{3}{5}\Big)^n\mathbb{E}_x\!\left[\#\text{ of visits of } X_t^n \text{ to }y\text{ before hitting }\partial \mathsf{T}\right]
    \!=\!\delta_n^{\beta-\alpha}\mathbb{E}_x\bigg[\sum_{k=0}^{\tau_{\mathsf{T}\backslash\partial \mathsf{T}}^n-1}\!\!\!\!\!\!\mathds{1}_{\{X^n_k=y\}}\!\bigg],
\end{align}
where $\tau_{\mathsf{T}\backslash\partial \mathsf{T}}^n$ is defined as in \eqref{definition:firFprovidst-exit-time} and $X^n_0=x$. Clearly $\greenFunctionDiscrete^n_{\mathsf{T}}(\cdot,y)=0$ for all $y\in \partial \mathsf{T}$. 
For any function $f:\mathsf{T}\rightarrow \R$  we define the Green operator $\greenOperatorDiscrete^n_{\mathsf{T}}$ as
$$ \greenOperatorDiscrete^n_{\mathsf{T}} f (x) := \delta_n^{\alpha} \sum_{y\in \mathsf{T}} \greenFunctionDiscrete^n_{\mathsf{T}}(x,y) f(y). $$
Since $\Delta_{\mathsf{SG}_n}\big( \greenFunctionDiscrete^n_{\mathsf{T}}(\cdot,y)\big) = -\delta_n^{\beta-\alpha}\dirac{y}{\cdot}$ we have $$ \Delta_n \greenOperatorDiscrete^n_{\mathsf{T}} f (x)= \greenOperatorDiscrete^n_{\mathsf{T}} \Delta_n f (x)=    -f (x), \quad \text{for all } x\in \mathsf{T}\backslash\partial\mathsf{T}.$$

\paragraph{Green functions on the Sierpi\'nski gasket $\SG$.}
As for its approximating counterpart, there is also no known globally defined Green function on $\mathsf{SG}$. 
Denote the ball in $\SG$ of radius $r\geq 0$ around $x$  in the Euclidean distance by $B(x,r):=\{y\in\SG: \ |x-y|\leq r\}$. We consider 
$$\domTriangle\subset\SG \quad \text{ as } \quad \domTriangle = B(0, 2^l),$$
where the radius $2^l$  of the ball will be specified later, with $l\in\N$. We want to emphasize here that in  introduction we kept the subscript $\domTriangle_l$, indicating the dependence on the radius $2^l$ of the ball. As we shall see below, the radius will not be important as long as it is big enough, reason for which we will mostly drop the subscript $l$. 
On $\domTriangle$ we define an energy form, a Laplacian, and a corresponding Green function.
We first define on $\domTriangle$ a Laplacian weakly via the bilinear energy form $\En_\domTriangle$, which itself is defined via Dirichlet forms on the approximating graphs $\SG_n$. More precisely, for functions $f,g:\domTriangle\cap \mathsf{SG}_n \rightarrow \R$ let
$$ \En^n_\domTriangle(f,g) := \Big( \frac{5}{3} \Big)^{n} \frac{1}{2}\sum_{x\sim_n y} \frac{1}{4}\big(f(x)-f(y)\big)\big(g(x)-g(y)\big),$$
where $x\sim_n y$ indicates that $x$ and $y$ are neighbors in $\mathsf{SG}_n$.
For the limiting object $\mathsf{SG}$ we first define an energy form on the dense subset $\SG_\infty$ by
\begin{align*}
	\Dom(\En^\infty_\domTriangle) &:= \left\{v:\mathsf{SG}_\infty\cap \domTriangle \to\R \text{ with }\ \lim_{n\rightarrow \infty} \En^n_\domTriangle(v|_{\mathsf{SG}_n\cap \domTriangle},v|_{\mathsf{SG}_n\cap \domTriangle}) < \infty \right\}, \\
	\En^\infty_\domTriangle(f,g) &:= \lim_{n\rightarrow\infty} \En^n_\domTriangle(f|_{\mathsf{SG}_n\cap \domTriangle},g|_{\mathsf{SG}_n \cap \domTriangle}) \ \text{for} \ f,g \in \Dom(\En^\infty_\domTriangle). 
\end{align*}
Note that $\Dom(\En^\infty_\domTriangle) \subseteq \mathcal{C}(\domTriangle)$, where $\mathcal{C}(\domTriangle)$ is the set of continuous functions on $\domTriangle$. We extend this definition to an energy form $\En_\domTriangle$ with domain $\Dom (\En_\domTriangle) = \{v\in \mathcal{C}(\domTriangle) :\  \En_\domTriangle(v,v) < \infty \}$ in a natural way.

Denote by 
\begin{align*}
    \Dom_0(\En_\domTriangle):=\{f\in \Dom(\En_\domTriangle):\ f|_{\partial\domTriangle}=0 \}
\end{align*} 
the set of all functions with finite energy and vanishing at the four boundary points of $\domTriangle$.
We define the set of double sided neighboring triangles for $v \in (\mathsf{SG}_\infty \cap \domTriangle)\backslash \partial\domTriangle$ of size $2^{-k}$ by 
\begin{align}\label{eq:triang-double}
    \doubleTriangle^k(v) := B(v,2^{-k}) \text{ for }k\geq \level(v).
\end{align}

Let $\mu$ be the rescaled $\alpha$-dimensional Hausdorff measure on $\mathsf{SG}$ such that $\mu(\domTriangle) = 3^{l}$. For the existence see \cite[Lemma 1.1]{barlow-perkins-BM}, with a slightly different rescaling adapted to our setup. The rescaling is chosen such that $\mu(\doubleTriangle^{0}(0))=1$ and therefore $$\mu(\doubleTriangle^{n}(0)) = \delta_n^\alpha.$$ 
\begin{defn}
Let $u\in \Dom(\En_\domTriangle)$ and $f:\mathsf{SG}\to\R$ be continuous. Then we define $u \in \Dom(\Delta_\domTriangle)$ with $\Delta_\domTriangle u = f$ if and only if
$$ \En_\domTriangle(u,v) = - \int_{\domTriangle} f v \ d\mu \text{ for all } v\in \Dom_0 (\En_\domTriangle). $$
We write $u\in \Dom_{L^2}(\Delta_\domTriangle)$ if we replace the assumption on the continuity of $f$ by $f\in L^2(\domTriangle, \mu)$. 
\end{defn}
See \cite[Definition 2.1.1]{strichartz-differential-equations-fractals} for more details on the definition of the weak Laplacian $\Delta_\domTriangle$.
\begin{lem}\label{lem:Conv-Laplacians}
	Let $u\in\Dom(\Delta_\domTriangle)$. Then the limit holds uniformly across $(\mathsf{SG}_\infty\cap \domTriangle) \backslash \partial\domTriangle$:
	\begin{align}
		\Delta_\domTriangle u(x) = \frac{3}{4}\lim_{n\to\infty} \Delta_n u(x).\label{Eq:LaplaceApprox}
	\end{align}
Conversely, if $u$ is a continuous function and the right hand side of \eqref{Eq:LaplaceApprox} converges uniformly to a continuous function on $(\mathsf{SG}_\infty\cap \domTriangle) \backslash \domTriangleBoundary$, then $u\in\Dom(\Delta_\domTriangle)$ and (\ref{Eq:LaplaceApprox}) holds.
\end{lem}
\begin{proof}
The claim follows from \cite[Theorem 2.2.1]{strichartz-differential-equations-fractals}, with the minor difference that in our case we have the factor $3/4$ instead of $3/2$, due to our choice of $\mu$, which allocates only half the mass on a unit copy of the gasket $\SG$ compared with the unit measure on this copy. Notice that in our definition of $\En_{\domTriangle}$ we have an additional factor of $1/4$ compared to \cite{strichartz-differential-equations-fractals} which comes from the fact that we use the probabilistic Laplacian throughout this paper, i.e. this factor is also resembled in $\Delta_n$.
\end{proof}

For $x,y \in \domTriangle = B(0, 2^l)\subset \SG$ we define the Green function $\greenFunction_\domTriangle:\domTriangle\rightarrow\R$ as in \cite{strichartz-differential-equations-fractals}. For $n\in\N$, define the level $n$ Green approximation $\greenFunction^n_\domTriangle$ on $\SG \cap \domTriangle$ by
\begin{align}
	\greenFunction^n_\domTriangle(x,y) &:=  r^{-l} \psi_0^{(-l)}(x)\psi_0^{(-l)}(y) + \sum_{m=-l}^{n-1} \ \sum_{z,z'\in \domTriangle\cap \mathsf{SG}_{m+1}\backslash \mathsf{SG}_m} g_m(z,z') \ \psi_z^{(m+1)}(x) \ \psi^{(m+1)}_{z'}(y) \\\label{eq:green-fc-sgn}
	&\text{for } g_m(z,z') = \begin{cases}
		\frac{18}{25} r^m \text{ if } z=z'\in \mathsf{SG}_{m+1}\backslash \mathsf{SG}_m \\
		\frac{6}{25} r^m \text{ if } z\sim_{m+1} z'\in \mathsf{SG}_{m+1}\backslash \mathsf{SG}_m \\
		0 \text{ otherwise}
	\end{cases}\nonumber
	\end{align}
and the Green function $ \greenFunction_\domTriangle $ on $\SG\cap \domTriangle$ as:
\begin{equation}\label{eq:green-fc-sg}
\greenFunction_\domTriangle(x,y):= \lim_{n\rightarrow\infty} \greenFunction^n_\domTriangle(x,y)
\end{equation}
where $r := 3/5$ is the energy scaling factor and $\psi_z^{(m)}(x)$ for $z\in \mathsf{SG}_m$ is called a harmonic spline and defined by the unique piecewise harmonic function on $\SG$ given by
$$\psi_z^{(m)}(x) = \begin{cases}
	1, \text{ if } x=z \\
	0, \text{ if } x\in \mathsf{SG}_m\backslash\{z\}.
\end{cases} $$
Notice that the constants in $g(z,z')$ slightly differ from the ones  \cite{strichartz-differential-equations-fractals} and \cite{kigami1989}, due to us working with the probabilistic Laplacian and the corresponding factor $1/4$ in the definition of $\En_\domTriangle$.

\begin{lem}\label{lem:green-func-equality}
For $x,y \in \mathsf{SG}_{n} \cap \domTriangle$ and for all $m\geq n$ it holds $$\greenFunction^n_\domTriangle(x,y) = \greenFunction^m_\domTriangle(x,y) = \greenFunction_\domTriangle(x,y).$$
\end{lem}
\begin{proof}
	The key insight is that $\psi_z^{(m+1)}(x) \ \psi^{(m+1)}_{z'}(y) \neq 0$ only if $x\in \doubleTriangle^{m+1}(z)$ and $y\in \doubleTriangle^{m+1}(z')$. But this cannot hold for $z, z' \in \mathsf{SG}_{m+1}\backslash \mathsf{SG}_m$ where $m\geq n$, and this yields
	\begin{align*}
		\greenFunction_\domTriangle(x,y) = \greenFunction^n_\domTriangle(x,y) + \sum_{m=n}^\infty \ \sum_{z,z'\in \SG_{m+1}\backslash \SG_m} g_m(z,z') \ \underbrace{\psi_z^{(m+1)}(x) \ \psi^{(m+1)}_{z'}(y)}_{=0}.
	\end{align*}
	
\end{proof}
\begin{lem}
For any $x,y\in \mathsf{SG}_{n}\cap \domTriangle \backslash \domTriangleBoundary$ and $n\in\N$ it holds
	$$ \Delta_{n} \greenFunction_\domTriangle(x,\cdot) (y) = - \delta_{n}^{-\alpha} \dirac{x}{y}. $$
\end{lem}
\begin{proof}
The rescaled discrete Laplacian $\Delta_n$ is connected to the energy in the  following way: since $\psi_y^{(m)}$ is piecewise harmonic we get
$$ \En_\domTriangle(u,\psi^{(m)}_y) = \En^{m}_\domTriangle(u,\psi^{(m)}_y) = -\delta_m^{\alpha-\beta}\sum_{x\in\SG_m\cap\domTriangle}\Delta_{\mathsf{SG}_m}u(x) \psi_y^{(m)}(x) = -\delta_m^\alpha \Delta_m u (y), $$
which together with \cite[Equation (2.6.17)]{strichartz-differential-equations-fractals} implies
	\begin{align*}
		\En_\domTriangle(\greenFunction^n_\domTriangle(x,\cdot),v) = \sum_{z\in \mathsf{SG}_{n}\cap \domTriangle \backslash \domTriangleBoundary} v(z) \psi^{(n)}_z(x) \quad \text{ for any } v\in \Dom_0(\En_\domTriangle).
	\end{align*} 
Applying the above two equations to $u=\greenFunction_\domTriangle(x,\cdot)$ and $v = \psi^{(n)}_y$ we finally obtain
	\begin{align*}
		\Delta_{n}\greenFunction_\domTriangle(x,\cdot) (y) &= - \delta_{n}^{-\alpha} \En_\domTriangle\left(G_\domTriangle(x,\cdot),\psi^{(n)}_y\right)
		= -  \delta_{n}^{-\alpha} \sum_{z\in \mathsf{SG}_{n}\backslash \domTriangleBoundary} \psi^{(n)}_y(z) \psi^{(n)}_z(x) = -  \delta_{n}^{-\alpha} \dirac{x}{y}.
	\end{align*}
\end{proof}
The last claim of this section relates the stopped Green function $\greenFunctionDiscrete_{\domTriangle}^n(x,y)$ defined in \eqref{eq:green-visits} with the Green functions defined in \eqref{eq:green-fc-sg} in terms of harmonic splines.
\begin{lem}\label{lem:connectionGreenFunctions}
For any $x,y \in \mathsf{SG}_n$ it holds
	\begin{align*}
		\greenFunctionDiscrete_{\domTriangle}^n(x,y) = \greenFunction_\domTriangle(x,y).
	\end{align*}
\end{lem}

\begin{proof}
Since $\Delta_{n} \greenFunctionDiscrete_{\domTriangle}^n(x,y) = -\delta_n^{-\alpha}\dirac{x}{y} $ for all $x,y\in \mathsf{SG}_{n}\cap \domTriangle\backslash \domTriangleBoundary$ and $ \greenFunctionDiscrete^n_{\domTriangle}(x,z) = 0$ for $z\in \domTriangleBoundary$, it follows that both functions solve the following  Dirichlet problem
	\begin{align*}
		\begin{cases}
			\Delta_n u(y) = -\delta_n^{-\alpha}\dirac{x}{y}\text{, if } y\in \mathsf{SG}_{n}\cap \domTriangle \backslash \domTriangleBoundary \\
			u(y) = 0 \text{, if } y \in \domTriangleBoundary
		\end{cases}
	\end{align*}
which together with the uniqueness principle implies $\greenFunctionDiscrete^n_{\domTriangle}(x,y) =  \greenFunction_\domTriangle(x,y).$
\end{proof}

\section{Scaling limit for the divisible sandpile}\label{sec:div-sand}

In this section we investigate odometer functions and obstacle problems on the limiting fractal object $\SG$. There are several
difficulties appearing on $\SG$.
We could define the pointwise Laplacian on $\SG$ and demand the odometer function $u:\SG\to\mathbb{R}$ to be the smallest function that fulfills $ \sigma +\Delta u\leq 1$,
where $\sigma$ is the initial mass on $\mathsf{SG}$. However, this definition is not very handy and it is not clear how to obtain any meaningful properties of it. We choose a different way to define {\it the continuous divisible sandpile on the gasket $\SG$} and motivate our choice below. 

\textbf{Passing from $\SG_n$ to $\SG$ and back: some useful notation.}
Recall that by $\doubleTriangle^{k}(x) $ we have denoted the double sided triangles of length $2^{-k}$ with midpoint $x\in \mathsf{SG}_n$. Similarly, the notations provided in Table \ref{tab:Notation} might be useful when working with points, sets and functions. However, they will only be used when it is absolutely clear on which approximating graph we are working. In the ambiguous case $x\in \mathsf{SG}_{n+1}\backslash \mathsf{SG}_n$ we define $x^\tridots$ to be rightmost neighbor.

\begin{table}[]
    \centering
    \renewcommand{\arraystretch}{1.3}
    \begin{tabular}{r|ll}
     & $\mathsf{SG}_n$ & $\mathsf{SG}$ \\ \hline
    points & $x^\tridots = {B\left(x,2^{-(n+1)}\right)}\cap \mathsf{SG}_n$ & $x^\doubleTriangle = \doubleTriangle^{n+1}(x)$ \\
    sets & $A^\tridots = A\cap \mathsf{SG}_n$ & $A^\doubleTriangle = \bigcup_{x\in A} x^\doubleTriangle$ \\
    functions & $f^\tridots = f|_{\mathsf{SG}_n}$ & $f^\doubleTriangle(x) = f(x^\tridots)$
    \end{tabular}
    \caption{Notation for transitioning between gasket $\SG$ and approximating graphs $\SG_n$.}
    \label{tab:Notation}
\end{table}
In order to motivate a meaningful definition, we would like to comment on what is needed when solving the divisible sandpile and the corresponding obstacle problem on $\mathsf{SG}_n$, for any $n\in\N$. We consider initial configurations $\sigma:\mathsf{SG}\rightarrow[0,\infty)$ and $\sigma_n:\mathsf{SG}_n\rightarrow[0,\infty)$ that fulfill the following properties: for a number $M\geq 0$ and a compact subset $\mathsf{T}\subseteq \mathsf{SG}$ it holds
\begin{align}
    &0 \leq \sigma, \sigma_n \leq M, \label{cond:Starting1} \\ 
    &\mu(DC(\sigma)) = 0, \label{cond:Starting2} \\ 
    &\sigma_n^{\doubleTriangle}(x) \rightarrow \sigma(x), \ \forall x\notin DC(\sigma)\text{ pointwise}, 
    \label{cond:Starting3}\\
    &\supp(\sigma),\ \supp(\sigma_n)\subseteq \mathsf{T}\label{cond:Starting4},
\end{align}
where by $DC(\sigma)$ we denote the set of discontinuity points of $\sigma$, and $\mu$ is the
($\alpha$-rescaled) Hausdorff measure on $\SG$.
For the rest of this section we denote the rescaled odometer function
$u_n$ of the divisible sandpile on $\SG_n$ from initial distribution $\sigma_n$ by
\begin{equation}\label{eq:rescaled-odom-n}
    u_n(x)=\delta_n^{\beta}\cdot(\text{total mass emitted from $x$ in the stabilization of }\sigma_n).
\end{equation}
We remark that the rescaled odometer function $u_n$ fulfills the {\it least action principle} as well, i.e. it is the smallest positive function such that $\sigma_n+\Delta_n u_n\leq 1$.
The following lemma provides a ball large enough containing the supports of all $u_n$, $n\in\N$.
\begin{lem}\label{lem:ball_contains_cluster_sandpile}
Let $\DD_n=\{u_n>0\}$ be the divisible sandpile cluster on $\SG_n$ started from initial density $\sigma_n$, for all $n\in\N$. Then there exists a triangle $\domTriangle= B(0,2^L)\subset \SG, L\in \N $ such that 
\begin{align*}
    \bigcup_{n\geq 0}\DD_n^{\doubleTriangle}\subseteq \domTriangle.
\end{align*}
\end{lem}
\begin{proof}
	Let $B=B(0,2^l)\subset \SG$ and choose $l\in \N $ large enough such that $B$ contains the supports of all $\sigma_n$.
Let $\mathcal{A}_n$ be the divisible sandpile cluster on $\mathsf{SG}_n$ from initial density $\sigma_n' =2|B\cap \SG_n|\mathds{1}_{\{0\}}$, which, by the inner bound of \cite{Div-Sandpile-SG}, contains $B$.
Let $\mathcal{B}_n$ be the divisible sandpile cluster on $\mathsf{SG}_n$  resulting from initial density $\tau = M\mathds{1}_{B}$. By the Abelian property we have $\DD_n\subseteq \mathcal{B}_n$, for all $n\in\N$ and also that the clusters $\mathcal{C}_n$ resulting from initial density $M\sigma_n'$ contain $\mathcal{B}_n$. Finally, choosing $L\in\N$ large enough such that $\domTriangle = B(0,2^L)$  has volume at least $3M \mu(B)$, the outer bound of \cite{Div-Sandpile-SG} yields that $\mathcal{C}_n\subseteq \domTriangle$, for all $n\in \N$.
\end{proof}

For the initial particle configuration $\sigma_n$ on $\mathsf{SG}_n$, we denote by $\nu_n$ the limit mass configuration obtained by stabilizing $\sigma_n$, and by $u_n$ the corresponding rescaled odometer function as defined in \eqref{eq:rescaled-odom-n} on $\SG_n$. Let $\domTriangle$ be a double sided triangle around $0$ as in Lemma \ref{lem:ball_contains_cluster_sandpile}. We define the obstacle function $\gamma_n:\SG_n\to\R$  and its least superharmonic majorant $s_n:\SG_n\to\R$  as
\begin{align}
    &\gamma_n(x) := - \sum_{y \in \mathsf{SG}_n \cap\domTriangle} \greenFunctionDiscrete^n_\domTriangle(x,y) (\sigma_n(y)-1), \text{ and } \label{eq:obst-div-sand} \\
    &s_n(x) := \inf\{f_n(x)|\ f_n:\mathsf{SG}_n\rightarrow\R \text{ is superharmonic on $\domTriangle$ and } f\geq \gamma_n \}. \nonumber
\end{align}
Then \cite[Lemma 3.9]{Div-Sandpile-SG} gives for the odometer function $u_n$  on $\SG_n$, $u_n := s_n - \gamma_n$,
and the sandpile cluster $\DD_n$ on $\SG_n$ starting from the initial configuration $\sigma_n$ 
 $$ \DD_n = \{x \in \mathsf{SG}_n:\  s_n(x) > \gamma_n(x) \},\quad \text{for all } n\in\N,$$
which motivates the following definition of divisible sandpile  on the Sierpi\'nski gasket $\SG$, by using the ball $\domTriangle$ (double sided triangle around $0$) of radius big enough containing all the divisible sandpile clusters $\DD_n$.
\begin{defn}[continuous divisible sandpile on $\domTriangle$]\label{def:contSandpile}
 Let $\sigma:\mathsf{SG} \rightarrow [0,\infty)$ be a bounded function with compact support and continuous $\mu$-a.e~such that $\textnormal{supp}(\sigma)\subseteq \domTriangle$, where $\domTriangle$ is a double sided triangle around $0$. We define the obstacle  $\gamma_{\domTriangle}:\SG\to\R$ by $\gamma_\domTriangle^{}=-\greenOperator_\domTriangle(\sigma-1)$:
 $$\gamma_{\domTriangle}(x) := - \int_\domTriangle \greenFunction_\domTriangle(x,y) (\sigma(y)-1) d\mu(y)$$
and the corresponding  odometer function $\uu_{\domTriangle}:\SG\to\R$ by
    \begin{align*}
        \uu_{\domTriangle} := s_{\domTriangle} - \gamma_{\domTriangle}, \text{ with } s_{\domTriangle}(x) := \inf\{ f(x) :\  f \text{ continuous, superharmonic on $\domTriangle$ and } f\geq \gamma_{\domTriangle} \}
    \end{align*}
being the smallest superharmonic majorant of $\gamma_{\domTriangle}$. Furthermore we define the noncoincidence set for the obstacle problem with obstacle $\gamma_{\domTriangle}$ as the domain
    $$D_{\domTriangle} = \{x \in \mathsf{SG}:\  s_{\domTriangle}(x) > \gamma_{\domTriangle}(x) \} .$$
\end{defn}
\textbf{Remark.} It is not a priori clear that the smallest superharmonic majorant $s_{\domTriangle}$ of $\gamma_{\domTriangle}$ as defined above is itself again superharmonic, or even in the domain of the Laplacian. However, using the fact that $\Dom_0(\En_\domTriangle)$ can be equipped with an inner product such that it becomes a Hilbert space ($\langle u, v \rangle_{\En}=\En_{\domTriangle}(u,v)+\langle u, v\rangle_{L^2(\domTriangle)}$) together with the Poincaré-type inequality
\begin{align*}|u(x)-u(y)|^2\leq C\En_\domTriangle(u)|x-y|^{\beta/\alpha},\end{align*}
for some constant $C>0$ as in \cite{strichartz-differential-equations-fractals}, we can deduce the superharmonicity of the superharmonic majorant as on $\R^d$. We can then use minimization theorems for coercive  functionals on Hilbert spaces in order to complete the claim, but since this is very similar to the case of $\R^d$, we have decided not to include it here. For more details we refer the interested reader to  \cite{sup-harm-maj} and \cite{free-pde-book}.

We aim at taking the limit $u=\lim_{k\rightarrow\infty}\uu_{\doubleTriangle^{-k}(0)}$ as the odometer function of the divisible sandpile on the infinite gasket $\SG$, but it is not straightforward why this limit should be well-defined. We show below that this is indeed the case. We first need some auxiliary results, the first one concerning the monotonicity of the model: by starting with more mass, one obtains a larger odometer and a larger noncoincidence set.

\begin{lem}\label{lem:monotonicity-if-div-sandpile}
Let $\sigma_1,\sigma_2:\SG\to[0,\infty)$ be bounded with compact support and $\sigma_1\leq \sigma_2$.
Let $\domTriangle$ be a double sided triangle with $\supp(\sigma_1),\supp(\sigma_2) \subseteq \domTriangle$ and, for $i=1,2$ consider the obstacle
\begin{align*}
    \gamma_\domTriangle^{(i)}=-\greenOperator_\domTriangle(\sigma_i-1).
\end{align*}
Let $s_\domTriangle^{(i)}$ be the least superharmonic majorant of $\gamma_\domTriangle^{(i)}$ and $\uu_\domTriangle^{(i)}=s_\domTriangle^{(i)}-\gamma_\domTriangle^{(i)}$. Then $\uu_\domTriangle^{(1)}\leq \uu_\domTriangle^{(2)}$.
\end{lem}
The proof is similar to  \cite[Lemma 2.8]{LP-mult-idla}, applied to our choice of odometers.

\begin{lem}\label{lem:odometer-stays-same}
Let $\sigma:\SG\to[0,\infty)$ as in Definition \ref{def:contSandpile}, and let $\domTriangle\subseteq \tilde{\domTriangle}$  be double sided triangles. Denote by $u:=\uu_{\domTriangle}$, $\tilde{u}:=\uu_{\tilde{\domTriangle}}$   the corresponding odometer functions  as in Definition \ref{def:contSandpile}. If there exists $\epsilon> 0 $ such that $D^\epsilon_{{\domTriangle}} \subsetneq \domTriangle$, then $\tilde{u} = u$.
\end{lem}
\begin{proof}
Let us denote by $\gamma:=\gamma_{\domTriangle}$ and $s:=s_{\domTriangle}$ (resp.  $\tilde{\gamma}:=\gamma_{\tilde{\domTriangle}}$ and  $\tilde{s}:=s_{\tilde{\domTriangle}}$) the obstacle and its superharmonic majorant on $\domTriangle$ (resp $\tilde{\domTriangle}$). We have $\Delta_\domTriangle\tilde{\gamma}=(1-\sigma)=\Delta_\domTriangle\gamma$. Notice that $\gamma$ and $s$ vanish on $\tilde{\domTriangle}\backslash\domTriangle$ and the function $s'=s+\tilde{\gamma}-\gamma$ is superharmonic on $\tilde{\domTriangle}$. Since, $s'\geq\tilde{\gamma}$, we obtain that $\tilde{s}\leq s'$ on $\tilde{\domTriangle}$.
From $D^\varepsilon_{{\domTriangle}}\subset\domTriangle$, it then follows $s'=\tilde{\gamma}$ in $\tilde{\domTriangle}\backslash\domTriangle$, which implies $\tilde{s}=\tilde{\gamma}$ on $\tilde{\domTriangle}\backslash\domTriangle$, which in turn gives $\tilde{u} = 0$ on $\tilde{\domTriangle}\backslash\domTriangle$. 
Inside $\domTriangle$ we can define the harmonic function $h=\tilde{\gamma}|_\domTriangle - \gamma|_\domTriangle$. By the definition of $s$ and $\tilde{s}$ we have for $x\in \domTriangle$
\begin{align*}
    \tilde{s}-h = \tilde{s} - \tilde{\gamma} + \gamma \geq \gamma \quad \text{and} \quad
    s + h = s + \tilde{\gamma} -\gamma \geq \tilde{\gamma},
\end{align*}
which implies $\tilde{s}(x) \geq s(x) + h(x) \geq \tilde{s}(x)$. Therefore, on $\domTriangle$ it holds
$\tilde{u} = \tilde{s}-\tilde{\gamma} = s - \gamma = u$,
and this proves the claim.
\end{proof}

\begin{lem}\label{lem:existence-of-global-sandpile}
Let $\sigma:\mathsf{SG}\rightarrow[0,\infty)$ be a bounded function with compact support and continuous $\mu$-almost everywhere. Denote by $\uu_{\doubleTriangle^{-k}(0)}$ the odometer of $\sigma$ on the set $\domTriangle=\doubleTriangle^{-k}(0)$ as  in Definition \ref{def:contSandpile}. Then the following limit is well-defined: 
\vspace{-0.25cm}
\begin{align*}
 u=\lim_{k\rightarrow\infty}\uu_{\doubleTriangle^{-k}(0)}.
\end{align*}
\end{lem}
\begin{proof}
We choose $l\in\N$ large enough such that $\text{supp}(\sigma)\subseteq B(0,2^l)$ and $L\in \N$ large enough such that $\sigma(x)\leq 3^{L-l}$, for all $x\in \SG$. Defining the initial configuration $\sigma':\SG\to[0,\infty)$ by
\begin{align*}
    \sigma'=\frac{\mu(B(0,2^L))}{\mu(B(0,2^l)}\mathds{1}_{B(0,2^l))}=3^{L-l}\mathds{1}_{B(0,2^l)},
\end{align*}
we then have $\sigma\leq\sigma'$. Let us denote by $\uu_{\doubleTriangle^{-k}(0)}$ the odometer of $\sigma$ in $\doubleTriangle^{-k}(0)$ and by $\uu_{\doubleTriangle^{-k}(0)}'$ the odometer of $\sigma'$ in $\doubleTriangle^{-k}(0)$. Appendix \ref{sec:appendixB} shows the existence of a compact set $\domTriangle=B(0,2^{L+1})$ such that for all $k\in \N$ with $\domTriangle\subseteq\doubleTriangle^{-k}(0)$ we have $\uu_{\doubleTriangle^{-k}(0)}'=0$ on $\doubleTriangle^{-k}(0)\setminus\domTriangle$, which together with Lemma \ref{lem:monotonicity-if-div-sandpile} implies that also $\uu_{\doubleTriangle^{-k}(0)}=0$ on  $\doubleTriangle^{-k}(0)\setminus\domTriangle$. The claim now follows from Lemma \ref{lem:odometer-stays-same}.
\end{proof}
We remark that Lemma \ref{lem:existence-of-global-sandpile} is actually a refined 
version of Proposition \ref{prop:obst1}.
We now  define the {\it continuous divisible sandpile} on  $\SG\subset\R^2$.
\begin{defn}[continuous divisible sandpile on $\mathsf{SG}$]\label{defn:global-cont-div-sand}
Let $\sigma:\mathsf{SG}\rightarrow[0,\infty)$ be bounded with compact support and continuous $\mu$-a.e. Denote by $\uu_{\doubleTriangle^{-k}(0)}=s_{\doubleTriangle^{-k}(0)}-\gamma_{\doubleTriangle^{-k}(0)}$ the odometer function of the continuous divisible sandpile in $\doubleTriangle^{-k}(0)$ as in Definition \ref{def:contSandpile}. Then the odometer function of the global continuous divisible sandpile is given by
\begin{align*}
    u=\lim_{k\rightarrow\infty}\uu_{\doubleTriangle^{-k}(0)},
\end{align*}
and we define the corresponding noncoincidence set $D$ by $D:=\{x\in \mathsf{SG}:\ u(x)>0\}$.
\end{defn}

\subsection{Convergence of odometers in divisible sandpile}
We now investigate the convergence of the odometer functions $u$ and $u_n$ with initial distributions $\sigma:\mathsf{SG}\rightarrow\R$ and $\sigma_n:\mathsf{SG}_n\rightarrow \R$, respectively, fulfilling the properties (\ref{cond:Starting1})-(\ref{cond:Starting4}). In view of Lemma \ref{lem:existence-of-global-sandpile} and Lemma \ref{lem:ball_contains_cluster_sandpile}, we can fix a large enough double sided triangle $\domTriangle$ that contains the supports of the odometer functions $u$ and $u_n$ for all $n\in\N$.
For the rest, in order to simplify the notation we drop the double sided triangle $\domTriangle$ and write 
\begin{align}
    \greenFunction := \greenFunction_\domTriangle \quad \greenFunctionDiscrete^n := \greenFunctionDiscrete^n_\domTriangle
\end{align}
for the Green function $\greenFunction$ on $\SG$, and for the Green function $\greenFunctionDiscrete$ stopped when exiting the set $\domTriangle$ in $\SG_n$, for all $n\in\N$.
\begin{lem}\label{lem:greenOperatorConvergence}
    Define $$\greenOperator \sigma (x):= \int_\domTriangle \greenFunction(x,y) \sigma(y) d\mu(y) \quad \text{and}\quad  \greenOperatorDiscrete^n \sigma_n (x):= \delta_n^\alpha \sum_{y\in \mathsf{SG}_n \cap \domTriangle} \greenFunctionDiscrete^n(x,y) \sigma_n(y).$$
If $\sigma,\sigma_n$ satisfy equations \eqref{cond:Starting1}-\eqref{cond:Starting4}, then
    $$ (\greenOperatorDiscrete^n\sigma_n)^{\doubleTriangle} \rightarrow \greenOperator\sigma \text{ uniformly as } n\to\infty.$$  
\end{lem}
\begin{proof}
By the triangle inequality we have
$$ |  (\greenOperatorDiscrete^n\sigma_n)^{\doubleTriangle} - \greenOperator\sigma| \leq | (\greenOperatorDiscrete^n\sigma_n)^{\doubleTriangle}- \greenOperator\sigma_n^{\doubleTriangle} | + | \greenOperator \sigma_n^{\doubleTriangle} - \greenOperator \sigma |.$$
The second term above converges uniformly to zero by the dominated convergence theorem since $\greenFunction$ is bounded and by assumptions (\ref{cond:Starting2}) and  (\ref{cond:Starting3}) $\sigma_n^{\doubleTriangle}(x) \rightarrow \sigma(x)$, $\mu$-almost everywhere as $n\to\infty$. The first term is at most
\begin{align}
     \big| (\greenOperatorDiscrete^n\sigma_n)^{\doubleTriangle} (x)- \greenOperator\sigma_n^{\doubleTriangle} (x)\big| 
    =   \bigg| \sum_{y\in \mathsf{SG}_n\cap \domTriangle} \sigma_n(y) \int_{y^\doubleTriangle} \Big(\greenFunctionDiscrete^n(x^\tridots,y) - \greenFunction(x,z)  \Big)d\mu(z) \bigg| \label{eq:proof:greenOperatorConv:1}
\end{align}
For $y\neq x^\tridots$ the function $f(z):= \greenFunctionDiscrete_n(x^\tridots,y)-\greenFunction(x^\tridots,z)$ is harmonic on $y^\doubleTriangle$, and Lemma \ref{lem:connectionGreenFunctions} gives $f(y)=0$ which yields
$$ \int_{y^\doubleTriangle}  \big(\greenFunctionDiscrete^n(x^\tridots,y) - \greenFunction(x^\tridots,z)  \big)d\mu(z) = \int_{y^\doubleTriangle} f(z) d\mu(z) = f(y) = 0.$$
So the sum over all $y\neq x^\tridots$ in \eqref{eq:proof:greenOperatorConv:1} can be bounded by
\begin{align*}
    &\Big| \sum_{y\in \mathsf{SG}_n\cap \domTriangle \backslash\{x^\tridots\}} \sigma_n(y) \int_{y^\doubleTriangle} \big( \greenFunctionDiscrete^n(x^\tridots,y) - \greenFunction(x,z) \big) d\mu(z) \Big| \\
    &= \bigg| \sum_{y\in \mathsf{SG}_n\cap \domTriangle \backslash\{x^\tridots\}} \sigma_n(y) \int_{y^\doubleTriangle} (\greenFunction(x^\tridots,z) - \greenFunction(x,z)) d\mu(z) \bigg| < M \epsilon,
\end{align*}
for any $\epsilon>0$ and $n\in\N$ large, where the last inequality is due to the uniform continuity of $\greenFunction$ together with (\ref{cond:Starting1}).
Finally, for $n$ large the summand $y=x^\tridots$ in \eqref{eq:proof:greenOperatorConv:1} can be bounded by
\begin{align*}
    \Big| \sigma_n(x^\tridots) \int_{(x^\tridots)^\doubleTriangle} \Big( \greenFunctionDiscrete^n(x^\tridots,x^\tridots) - \greenFunction(x,z) \Big) d\mu(z) \Big| \leq C \mu(y^\doubleTriangle) \leq C \delta_n^\alpha < c \epsilon
\end{align*}
because $\greenFunction$ is bounded and $\greenFunctionDiscrete^n(y,y) \sim G(y,y)$ by Lemma \ref{lem:connectionGreenFunctions}, and this completes the proof.
\end{proof}

\begin{corollary}\label{cor:obstacleConvergence}
If $\sigma,\sigma_n$ satisfy (\ref{cond:Starting1})-(\ref{cond:Starting4}), then
$ \gamma_n^\doubleTriangle \rightarrow \gamma \text{ uniformly}$ as $n\to\infty$.
\end{corollary}

\begin{lem}\label{lem:convergenceSuperharmonicMajorant}
    If $\sigma,\sigma_n$ satisfy (\ref{cond:Starting1})-(\ref{cond:Starting4}), then
    $ s_n^\doubleTriangle \rightarrow s$  as  $n\rightarrow\infty$  uniformly on  $\domTriangle$.
\end{lem}

\begin{proof}

For $x\in\SG_n$, denote by $\psi_x^n$ the piecewise harmonic spline of level $n$ which is $1$ at $x$. We then have for the superharmonic majorant $s$ that
\begin{align*}
\Delta_n s(x)=\mathcal{E}_\domTriangle(s,\psi_x^{n})=-\int_\domTriangle \Delta_\domTriangle s\cdot \psi_x^n d\mu \leq 0,
\end{align*}
where in the last inequality we have used the superharmonicity of $s$ together with the non-negativity of $\psi_x^n$. This shows that $s$ restricted to $\SG_n\cap\domTriangle$ is also superharmonic. By Corollary \ref{cor:obstacleConvergence} we can find for any $\varepsilon>0$ an $n_0\in \N$ such that for all $n\geq n_0$ we have
\begin{align*}
\gamma_n^{\doubleTriangle}-\varepsilon\leq \gamma.
\end{align*}
We thus obtain that $s|_{\SG_n}+\varepsilon$ is superharmonic and
\begin{align*}
s|_{\SG_n}+\varepsilon\geq \gamma|_{\SG_n} + \varepsilon \geq \gamma_n,
\end{align*}
for all $n\geq n_0$. Thus, by the definition of $s_n$ as the least superharmonic majorant of $\gamma_n$ on $\SG_n$, we obtain for large enough $n$ that
\begin{align*}
s+\varepsilon \geq s_n^{\doubleTriangle}.
\end{align*}

For the reverse inequality consider $\nu_n$ the final sandpile distribution on $\mathsf{SG}_n$ for the initial density $\sigma_n$. Recall that $$\Delta_n s_n = \Delta_n u_n + \Delta_n \gamma_n = (\nu_n - \sigma_n ) + (\sigma_n -1 ) = \nu_n -1 $$
and therefore $0 \leq \Delta_n s_n \leq 1$.
By denoting $\phi_n := -\Delta_n s_n$ and $\tilde{s} = \greenOperator \phi_n^\doubleTriangle(x)$, we obtain for large enough $n\in \N$
\begin{align*}
    &|\tilde{s} - s_n^\doubleTriangle(x)| = |\greenOperator \phi_n^\doubleTriangle(x) - \greenOperatorDiscrete^n \phi_n(x^\tridots)| \\
    &\leq \sum_{y\in \SG_n\cap \domTriangle} |\phi_n(y)| \  \Big|\int_{\doubleTriangle^n(y)} \mu(\doubleTriangle^n(y))^{-1} \delta_n^\alpha\greenFunctionDiscrete(x^\tridots,y)- \greenFunction(x,z) d\mu(z)\Big| < \epsilon,
\end{align*}
where in the last step we have used  the same arguments as in the proof of Lemma \ref{lem:greenOperatorConvergence} in order to show that (\ref{eq:proof:greenOperatorConv:1}) goes to zero together with the fact $0 \leq \phi_n \leq 1$.
Thus
\begin{align*}
\tilde{s} + 2 \epsilon > s_n^\doubleTriangle(x) + \epsilon > \gamma_n^\doubleTriangle(x) + \epsilon > \gamma 
\end{align*}
and by construction $\tilde{s}$ is superharmonic, from which we can finally conclude that
$$ s_n^\doubleTriangle + 3 \epsilon \geq \tilde{s} + 2 \epsilon> s,$$
which finishes the proof.
\end{proof}

\begin{lem}\label{theo:convergenceSandpileOdometer}
    Let $u_n$ be the odometer function for the divisible sandpile on $\mathsf{SG}_n$ with initial density $\sigma_n$. If $\sigma,\sigma_n$ satisfy \eqref{cond:Starting1}-\eqref{cond:Starting4}, then
    \begin{align}
        u_n^\doubleTriangle \rightarrow s-\gamma = u \text{ uniformly on compacts, as } n\to\infty.
    \end{align}
\end{lem}
\begin{proof}
By Corollary \ref{cor:obstacleConvergence} and Lemma \ref{lem:convergenceSuperharmonicMajorant} we have $\gamma_n^\doubleTriangle \rightarrow \gamma$ and $s_n^\doubleTriangle \rightarrow s$ uniformly on $\domTriangle$, as $n\to\infty$. Using $u_n = s_n - \gamma_n$ and $u = s - \gamma$ gives the result.
\end{proof}

\subsection{Convergence of domains in divisible sandpile}

Following \cite{LP-mult-idla}, we assume here additional conditions on the initial density $\sigma:\SG\to[0,\infty)$:
\begin{align}
    &\text{ For all $x\in \mathsf{SG}$ either } \sigma(x) \leq \lambda \text{ or } \sigma(x)\geq 1 \text{ for some $\lambda < 1$,} \label{cond:additionalStartingCondition1}\\
    &\{\sigma \geq 1 \} = \overline{\{\sigma \geq 1 \}^\circ}.\label{cond:additionalStartingCondition2}
\end{align}
Moreover, we assume that for any $\epsilon>0$ there exists $N(\epsilon)$ such that
\begin{align}
    \text{if } x\in \{\sigma \geq 1 \}_\epsilon, \text{ then } \sigma_n(x) \geq 1 \text{ for all } n\geq N(\epsilon), \label{cond:additionalStartingCondition3}\\
    \text{if } x\notin \{\sigma \geq 1 \}^\epsilon, \text{ then } \sigma_n(x) \leq \lambda \text{ for all } n\geq N(\epsilon). \label{cond:additionalStartingCondition4}
\end{align}
We also write $$\widetilde{D} := D\cup \{ x\in \mathsf{SG} :\  \sigma(x) \geq 1\}^\circ. $$
The aim of this section is to prove the scaling limit for the divisible sandpile cluster $\DD_n$.
\begin{thm}\label{theo:convergenceSandpileDomains}
 Let $\sigma$ and $\sigma_n$ satisfy \eqref{cond:Starting1}-\eqref{cond:Starting4} and \eqref{cond:additionalStartingCondition1}-\eqref{cond:additionalStartingCondition4}. For $n\geq 1$, let $\DD_n$ be the divisible sandpile cluster in $\mathsf{SG}_n$ started from initial density $\sigma_n$. For any $\epsilon>0$ and $n$ large enough, we have
\begin{align*}
\tilde{D}_\epsilon \cap \mathsf{SG}_n \subseteq \DD_n \subseteq \tilde{D}^\epsilon \cap \mathsf{SG}_n.
\end{align*}
\end{thm}
\begin{lem}\label{lem:sandpileConvergenceDomains:odometerbound}
Fix $\epsilon >0$ and $x\in \DD_n$ with $x\notin \widetilde{D}^\epsilon$. For $n$ large enough, there is a vertex $y\in \mathsf{SG}_n$ with $|x-y|\leq \epsilon/2 + \delta_n$, such that for some constant $c>0$ it holds
$$ u_n(y) \geq u_n(x) + c (1-\lambda)\epsilon^\beta. $$

\end{lem}
\begin{proof}
From  (\ref{cond:additionalStartingCondition2}) we have $\{\sigma \geq 1\} \subseteq \overline{\widetilde{D}}$. Thus if $x\notin \widetilde{D}^\epsilon$, the ball $B= B(x,\frac{\epsilon}{2})^\tridots$, with $B(x,\frac{\epsilon}{2})\subset\SG$ is disjoint from $\{\sigma \geq 1 \}^{\epsilon/2}$. Taking $n\geq N(\frac{\epsilon}{2})$ gives  $\sigma_n \leq \lambda$ on $B$ by assumption (\ref{cond:additionalStartingCondition4}).
We define for $y\in \mathsf{SG}_n$
\begin{equation}\label{eqref:phi-expec}
\phi_{x,r}(y) := \E_x[\tau^n_r(x)]- \E_y[\tau^n_r(x)].
\end{equation}
We then have $\Delta_{\mathsf{SG}_n} \phi_{x,r}(y) = 1$ for $y\in B^n(x,r)$ and $\phi_{x,r}(x) = 0$. Since  $r$ is the radius in the graph metric, when translating to $B$ we have to rescale it by $\delta_n$. Define $w:\SG_n\to\R$ by
    $$ w(y) = u_n(y) - (1-\lambda) \delta_n^\beta \phi_{x,\epsilon/2\delta_n}(y), $$
which is subharmonic on $B\cap \DD_n$ and therefore attains its maximum on the boundary. By the choice of $\phi_{x,\epsilon/2\delta_n}$, $w(x) > 0$, so the maximum cannot be attained at $\partial \DD_n$, where $u_n$ vanishes. Therefore the maximum is attained at some $y\in \partial B$ and
    $$ u_n(y) \geq w(y) + (1-\lambda) \delta_n^\beta \phi_{x,\epsilon/2\delta_n}(y) \geq w(y) + (1-\lambda) c \delta_n^\beta \Big(\frac{\epsilon}{\delta_n}\Big)^\beta $$
    where the last inequality follows from \eqref{property:exit_time} together with $\E_y[\tau^n_r(x)] = 0$ for $y\in \partial B^n(x,r)$.
\end{proof}

\begin{proof}[Proof of Theorem \ref{theo:convergenceSandpileDomains}]
Fix $\epsilon >0$. By the definition of $\widetilde{D}_\epsilon $ we have for some $\eta>0$
    $$ \widetilde{D}_\epsilon \subseteq D_\eta \cup \{\sigma \geq 1 \}_\eta.$$
    Since the closure of $D_\eta$ is compact and contained in $D$, we have $u\geq m_\eta$ on $D_\eta$ for some $m_\eta > 0$. From Lemma \ref{theo:convergenceSandpileOdometer} we get $$u_n > u^\tridots - \frac{1}{2} m_\eta > 0 \text{ on } D_\eta \cap \mathsf{SG}_n \text{ for large enough }n,$$ hence $D_\eta \cap \mathsf{SG}_n \subseteq \DD_n$, which together with \eqref{cond:additionalStartingCondition3} gives $\{ \sigma \geq 1\}_\eta\cap \mathsf{SG}_n \subseteq \DD_n$, and this implies that $\tilde{D}_\epsilon \cap \mathsf{SG}_n \subseteq \DD_n$.
For the upper bound, fix $x\in \SG_n$ with $x\notin \widetilde{D}^\epsilon$. Since $u$ vanishes on $B=B(x,\frac{\epsilon}{2})$ we get from the convergence of $u_n$ in Lemma \ref{theo:convergenceSandpileOdometer} that $$ u_n < c(1-\lambda) \epsilon^\beta \text{ on } B\cap \mathsf{SG}_n \text{ for large enough } n,$$
which together with Lemma \ref{lem:sandpileConvergenceDomains:odometerbound} yields $x\notin \DD_n$ and therefore $\DD_n \subseteq \tilde{D}^\epsilon \cap \mathsf{SG}_n$.
\end{proof}
\section{Scaling limit for rotor aggregation}\label{sec:rotor}

For rotor aggregation, since particles are moving according to a rotor walk on $\SG_n$, the initial densities have to be given as integer valued functions. 
Similar to \cite{LP-mult-idla}, we adapt the notion of convergence of densities and we define the appropriate notions of odometer functions.

\textbf{Smoothing operation.} For more details on the smoothing operation and why it is appropriate to use it in order to define a notion of convergence, we refer the reader to \cite[Section 4]{LP-mult-idla}. For any $n\in\N$, denote in this section by $(Y^n_t)_{t\in\N}$ the lazy random walk on $\SG_n$, which at each time step stays at the current position with probability $1/2$ and with probability $1/8$ chooses one of the four possible neighbors and moves there.
For any function $f:\mathsf{SG}_n\rightarrow\R$ and $k\in\N$, we define its {\it $k$-th smoothing} by
\begin{align*}
    S^{n}_k f(x)=\sum_{y\in \mathsf{SG}_n}\mathbb{P}_x[Y^{n}_k=y]f(y)=\mathbb{E}_x[f(Y^{n}_k)],
\end{align*}
so the smoothing operation fulfills $S_k S_l=S_{k+l}$ for all $k,l\in\N$ due to the Markov property of the lazy random walk. Furthermore, we have
$\Delta_n=2\delta_n^{-\beta}(S_1-S_0)$,
which implies that the Laplacian and the smoothing operator commute. 

\textbf{Lazy random walk on Sierpi\'nski graphs $\SG_n$.}  
According to \cite[Proposition 5.7]{markov_and_mixing}, 
the total variation distance between the distribution $P^{t+1}$ at time  $t+1$ and the distribution $P^t$ at time $t$ of the lazy random walk $(Y^n_t)_{t\in\SG_n}$ can be bounded by: 
\begin{align}\label{eq:tv-pt}
    ||P^{t+1}(x,\cdot)-P^t(x,\cdot)||_{\text{TV}}\leq 1/\sqrt{t}.
\end{align}
Next we collect some properties of the lazy random walk on $\mathsf{SG}_n$.
\begin{lem}\label{lem:BoundTransitionProbabilities}
Let $P$ be the transition matrix of the lazy random walk  $(Y^n_t)_{t\in\N}$ on $\SG_n$. Then
\begin{enumerate}[(1)]
\setlength\itemsep{0em}
    \item $\Delta_{\mathsf{SG}_n} P^t(x,y)=2P^{t+1}(x,y)-2P^{t}(x,y)$.
    \item $P^t(x,y)\leq \frac{C}{t^{\alpha/\beta}}\exp\left(-c\big(d(x,y)^\beta/t\big)^{1/(\beta-1)}\right)$.
    \item $P^t(x,x)\geq \frac{C'}{t^{\alpha/\beta}}$.
\end{enumerate}
where $C',C,c>0$ are suitable constants.
\end{lem}
\begin{proof}
\textit{(1)} Using
\begin{align*}
    P^t(x,y)=\frac{1}{2}P^{t-1}(x,y)+\frac{1}{8}\sum_{y'\sim y}P^{t-1}(x,y'),
\end{align*}
we get for the first claim:
\begin{align*}
& \Delta_{\mathsf{SG}_n} P^t(x,y)=\frac{1}{4}\sum_{y'\sim y}P^t(x,y') - P^t(x,y)\\
 &=2\Big(\frac12 P^{t}(x,y)+\frac{1}{8}\sum_{y'\sim y}P^t(x,y')\Big) - 2P^t(x,y)
    =2P^{t+1}(x,y) - 2P^t(x,y).
\end{align*}
\textit{(2)} For the second and third claim we use the following bounds on the transition probabilities for the simple random walk $(X^n_t)_{t\in\N}$ from \cite[Theorem 17]{jones_transition_1996}:
\begin{align*}
\mathds{1}_{\{x=y\}}C_2\frac{1}{t^{\alpha/\beta}}\leq \mathbb{P}[X_t^n=y|X^n_0=x]\leq \frac{C_1}{t^{\alpha/\beta}}\exp\Big(-c_1\Big(\frac{d(x,y)^\beta}{t}\Big)^{1/(\beta-1)}\Big),
\end{align*}
for $C_1,c_1,C_2>0$. The $t$-steps transition probabilities of the lazy random walk $(Y^n_t)_{t\in\N}$ are then
\begin{align*}
    P^t(x,y)&=\sum_{k=0}^t{t \choose k}2^{-t}\mathbb{P}_x[X^n_k=y] = 2^{-t}\mathds{1}_{\{x=y\}} + \sum_{k=1}^t{t \choose k}2^{-t}\mathbb{P}_x[X^n_k=y] \\
    &\leq 2^{-t} + \exp\Big(-c_1\Big(\frac{d(x,y)^\beta}{t}\Big)^{1/(\beta-1)}\Big)\sum_{k=1}^t{t \choose k}2^{-t}\frac{C_1}{k^{\alpha/\beta}} .
\end{align*}
On the other hand
\begin{align*}
   \sum_{k=1}^t{t \choose k}2^{-t}\frac{t^{\alpha/\beta}}{k^{\alpha/\beta}} = \E\Big[\frac{t^{\alpha/\beta}}{Z^{\alpha/\beta}} \mathds{1}_{\{Z \geq 1\}}  \Big] \leq \E\Big[\frac{t}{Z} \mathds{1}_{\{Z \geq 1\}} \Big]  
    \leq 2 t \E\Big[ \frac{1}{Z+1} \Big] = 2 t \frac{2(1-2^{-{t+1}})}{t+1} \leq 4,
\end{align*}
where $Z\sim \text{Binom}(t,1/2)$, and this proves the second assertion. 

{\it (3)} Using the lower bound on the transition probabilities $\mathbb{P}_x[X^n_t=x]$ we get 
\begin{align*}
    P^t(x,x)=&\mathbb{P}_x[Y^n_t=x]=\sum_{k=0}^t{t \choose k}2^{-t}\mathbb{P}_x[X^n_k=x]
    \geq C_2\frac{1}{t^{\alpha/\beta}}\sum_{k=0}^t{t \choose k}2^{-t}
    =C_2\frac{1}{t^{\alpha/\beta}},
\end{align*}
which proves the final assertion.
\end{proof}
\begin{prop}\label{prop:total_variation_lazy_walk}
For the lazy random walk $(Y_t^n)_{t\geq 0}$ on $\mathsf{SG}_n$, for every $x\in\mathsf{SG}_n$ and $0<\delta<1/2\alpha$,
we have
\begin{align*}
\sum_y\sum_{y'\sim y}|P^t(x,y)-P^t(x,y')|\leq \frac{C}{t^{1/4-\delta\alpha/2}}.
\end{align*}
\end{prop}
\begin{proof}
Fix any $\delta>0$ with $\delta<1/2\alpha$. We bound the sum by considering it for elements inside and outside the ball $B^n(x,t^{1/\beta+\delta})$ separately. We have

\begin{align*}
 &\sum_{y\in B^n(x,t^{1/\beta+\delta})}\sum_{y'\sim y}|P^t(x,y)-P^t(x,y')|\\
 & \leq\Big| B^n(x,t^{1/\beta+\delta})\Big|^{1/2}\Big(\sum_{y\in B(x,t^{1/\beta+\delta})}\sum_{y'\sim y}|P^t(x,y)-P^t(x,y')|^2\Big)^{1/2}\\
 & \leq\Big| B^n(x,t^{1/\beta+\delta})\Big|^{1/2}\left(\mathcal{E}_{\SG_n}(P^t(x,\cdot),P^t(x,\cdot))\right)^{1/2},
\end{align*}
where the first inequality holds by Cauchy-Schwarz and the (unscaled) graph energy form $\mathcal{E}_{\SG_n}$ on $\SG_n$ is defined as $\mathcal{E}_{\SG_n}(f,g) = 1/2 \sum_{x\sim_n y} 1/4 (f(x)-f(y))(g(x)-g(y))$. By equation \eqref{eq:tv-pt} and Lemma \ref{lem:BoundTransitionProbabilities} we get
\begin{align*}
    |\mathcal{E}_{\SG_n}(P^t(x,\cdot),P^t(x,\cdot))|&=\Big|\sum_{y}P^t(x,y)\Delta_{\SG_n} P^t(x,y)\Big|
    \leq \frac{2C}{t^{\alpha/\beta}}\sum_{y}|P^t(x,y)-P^{t+1}(x,y)|
    \leq \frac{2C}{t^{\alpha/\beta+1/2}}.
\end{align*}
Because $\left|B^n(x,r)\right| \leq c r^{\alpha}$ 
\begin{align*}
    \sum_{y\in B^n(x,t^{1/\beta+\delta})}\sum_{y'\sim y}|P^t(x,y)-P^t(x,y')|&\leq \frac{C''}{t^{-\alpha/2\beta+1/4+\alpha/2\beta-\delta\alpha/2}}
    =\frac{C''}{t^{-\delta\alpha/2+1/4}}.
\end{align*}
For the remaining terms we have
\begin{align*}
    \sum_{y\notin B^n(x,t^{1/\beta+\delta})}\sum_{y'\sim y}|P^t(x,y)-P^t(x,y')|
    \leq C'''\left|B^n(x,t)\backslash B^n(x,t^{1/\beta+\delta})\right|\frac{1}{t^{\alpha/\beta}}\exp\left(-c(t^{\delta\beta})^{1/(\beta-1)}\right).
\end{align*}
Since $\left|B^n(x,t)\backslash B^n(x,t^{1/\beta+\delta})\right|$ grows polynomially in $t$, the dominating term is the exponential part in $t$, which goes to $0$ for $t$ large, and this completes the proof.
\end{proof}

\subsection{Convergence of odometers in rotor aggregation}\label{sec:RR-odometer}

Consider $\sigma:\SG\to[0,\infty)$ and $\sigma_n:\SG_n\to \N$ that satisfy equations
\eqref{cond:Starting1},\eqref{cond:Starting2} and \eqref{cond:Starting4}. As a replacement for \eqref{cond:Starting3}, in order to define a suitable convergence of $\sigma_n$ to $\sigma$, suppose there exist integers $\kappa(n)\rightarrow \infty$ with $\kappa(n)\cdot\delta_n\rightarrow 0$ such that
\begin{align}
    \left( S^{n}_{\kappa(n)}\sigma_n\right)^\doubleTriangle (x)\rightarrow \sigma(x), \label{cond:RROdometerConvergence1}
\end{align}
for all $x\notin DC(\sigma)$.
Since we write $\kappa(n)$ for the smoothing steps on $\mathsf{SG}_n$ we drop the superscript and write only $S_{\kappa(n)} = S^{n}_{\kappa(n)}$. Similar to \cite[Section 4]{LP-mult-idla}, we define {\it the odometer function for the rotor aggregation on $\SG_n$ starting from density $\sigma_n$} as the function $u_n:\SG_n\to\N_0$ with
\begin{align}\label{cond:RROdometerConvergence2}
 u_n(x)=\delta_n^{\beta}\cdot\#(\text{particles emitted from }x),
\end{align}
if $\sigma_n(y)$ particles start at each site $y$, and we remind that particles perform rotor walks until each of them has found a previously unvisited site where it stops.
\begin{thm}\label{thm:conv_odometer_rotor}
Let $u_n$ be the odometer function for rotor aggregation on $\mathsf{SG}_n$ with initial density $\sigma_n$ as defined in \eqref{cond:RROdometerConvergence2}. If $\sigma,\sigma_n$ satisfy \eqref{cond:Starting1},\eqref{cond:Starting2}, \eqref{cond:Starting4} and \eqref{cond:RROdometerConvergence1}, then $u_n\rightarrow u$ uniformly as $n\to\infty$, where $u$ is the odometer for the divisible sandpile on $\mathsf{SG}$ starting with $\sigma$.
\end{thm}
For a function $f:\mathsf{SG}_n\to\R$ and a directed edge $(x,y)$
we write
$ \nabla f(x,y)=\delta_n^{-\beta/2}(f(y)-f(x))$
and for a function $\theta$ defined on directed edges in $\mathsf{SG}_n$ we write
\begin{align*}
    \text{div}\theta(x)=\frac{1}{4\delta_n^{\beta/2}}\sum_{y\sim x}\theta(x,y).
\end{align*}
With this in mind, the rescaled discrete Laplacian on $\mathsf{SG}_n$ is then given by
\begin{align*}
    \Delta_n f(x)=\text{div}\nabla f(x).
\end{align*}
\begin{lem}\label{lem:divergence_of_odometer_rotor}
Denote by $\theta(x,y)$ the net number of crossings from $x$ to $y$ performed by particles during a sequence of rotor moves, where $(x,y)$ is some directed edge in $\mathsf{SG}_n$, and let $u$ be the odometer function for this sequence of rotor moves. Then
\begin{align*}
    \nabla u(x,y)=\delta_n^{\beta/2}(-4\theta(x,y)+\rho(x,y)),
\end{align*}
for some function $\rho$ on directed edges of $\mathsf{SG}_n$ satisfying $|\rho(x,y)|\leq 6$
for all directed edges.
\end{lem}
\begin{proof}
If $N(x,y)$ is the number of particles routed over the edge $(x,y)$ in direction $y$, we have for any $y,z\sim x$
\begin{align*}
    \theta(x,y) = N(x,y)-N(y,x)\quad \text{and}\quad  |N(x,y)-N(x,z)|\leq 1.
\end{align*}
By definition $u(x)=\delta_n^{\beta}\sum_{y\sim x}N(x,y)$ and we obtain
\begin{align*}
    \delta_n^{-\beta}u(x)-3\leq 4N(x,y) \leq \delta_n^{-\beta}u(x)+3.
\end{align*}
Finally, this yields
\begin{align*}
    |\nabla u(x,y)+4\delta_n^{\beta/2}\theta(x,y)|&=\delta_n^{\beta/2}|\delta_n^{-\beta}u(y)-\delta_n^{-\beta}u(x)+4N(x,y)-4N(y,x)|
    \leq 6\delta_n^{\beta/2}.
\end{align*}
\end{proof}
Let $\theta$ as in Lemma \ref{lem:divergence_of_odometer_rotor} and denote by $\RR_n$ the rotor cluster on $\SG_n$ with initial particles density $\sigma_n$. Clearly, since $\sigma_n(x)$ particles start and $\mathds{1}_{\RR_n}(x)$ particles end up in $x\in \SG_n$, we have 
\begin{align*}
	4\delta_n^{\beta/2}\text{div}\theta=\sigma_n-\mathds{1}_{\RR_n}.
\end{align*}
Using Lemma \ref{lem:divergence_of_odometer_rotor} and taking the divergence of $\nabla u_n$ gives
\begin{align*}
	\Delta_n u_n=\mathds{1}_{\RR_n}-\sigma_n+\delta_n^{\beta/2}\text{div}\rho.
\end{align*}
Having in mind that $0\leq \sigma_n \leq M$ and $|\rho|\leq 6$, we obtain
\begin{align}
	|\Delta_n u_n|\leq M+6 \quad \text{ on } \SG_n. \label{eq:RR-bound-odometer-Laplacian}
\end{align}
The next lemma gives an estimate on the Laplacian of the smoothed odometer $S_k u_n$, and it can be seen as the smoothed counterpart of the sandpile odometer having Laplacian $1-\sigma_n$.
\begin{lem}\label{lem:Laplacian-of-smoothed-odometer}
For the lazy random walk $(Y^n_t)_{t\in\N}$  on $\SG_n$, there is a constant $C_0>0$ such that for every $n\in\N$
for any $0<\delta<1/2\alpha$ it holds:
\begin{align*}
    |\Delta_n S_ku_n(x)-\mathbb{P}_x[Y^n_k\in \RR_n]+S_k\sigma_n(x)|\leq \frac{C_0}{k^{1/4-\delta\alpha/2}}.
\end{align*}
\end{lem}
\begin{proof}
Let $\theta$ and $\rho$ as in Lemma \ref{lem:divergence_of_odometer_rotor}. Since $S_k$ and $\Delta_n$ commute:
\begin{align*}
    \Delta_n S_ku_n(x)=\mathbb{P}_x[Y_k\in \RR_n]-S_k\sigma_n(x)+\delta_n^{\beta/2}S_k\text{div}\rho(x).
\end{align*}
Because $\theta$ and $\nabla u$ are antisymmetric, so is $\rho$ by Lemma \ref{lem:divergence_of_odometer_rotor}. For $(Y^n_t)$ on $\SG_n$, we have
\begin{align*}
    \delta_n^{\beta/2}S_k\text{div}\rho(x)&=\frac{1}{4}\sum_{y\in\SG_n}\sum_{z\sim y}\mathbb{P}_x[Y^n_k=y]\rho(y,z)
    =-\frac{1}{4}\sum_{y\in\SG_n}\sum_{z\sim y}\mathbb{P}_x[Y^n_k=z]\rho(y,z)\\
    &=\frac{1}{8}\sum_{y\in\SG_n}\sum_{z\sim y}(\mathbb{P}_x[Y^n_k=y]-\mathbb{P}_x[Y^n_k=z])\rho(y,z).
\end{align*}
Taking the absolute value and using the triangle inequality as well as the fact that $\rho$ is bounded, together with the bound in Proposition \ref{prop:total_variation_lazy_walk}, the claim follows.
\end{proof}
\begin{lem}\label{lem:smoothing_odometer_error_bound}
	For $0<\delta<1/2\alpha$ and the odometer function $u_n$ on $\SG_n$ as in \eqref{cond:RROdometerConvergence2}, we have
	\begin{align*}
	|S_k u_n-u_n|\leq \delta_n^{\beta}\Big(\frac{1}{2}(M+6)k+C_0k^{3/4+\delta\alpha/2}\Big).
	\end{align*}
\end{lem}
\begin{proof}
	It holds
	\begin{align*}
	    |S_ku_n-u_n|&\leq \sum_{j=0}^{k-1}|S_{j+1}u_n-S_ju_n|=\frac{\delta_n^\beta}{2}\sum_{j=0}^{k-1}|\Delta_n S_j u_n|.
	\end{align*}
	Lemma \ref{lem:Laplacian-of-smoothed-odometer} for $0<\delta<1/2\alpha$ and $j\geq 1$ gives
	\begin{align*}
	    |\Delta_n S_j u_n|\leq M+1+\frac{C_0}{j^{1/4-\delta\alpha/2}},
	\end{align*}
	and for $j=0$ we have $|\Delta_n S_j u_n|=|\Delta_n u_n|\leq M+6$. Summing over $j$ yields the result.
\end{proof}
Now let $\domTriangle = B(0,2^L)\subset \SG$ be the double triangle from Lemma \ref{lem:ball_contains_cluster_sandpile} that contains the final domains of occupied sites for the divisible sandpiles started from $S_{\kappa(n)}\sigma_n$ on $\mathsf{SG}_n$. In the following, when referring to the Green function, we mean the Green function stopped when exiting $\domTriangle^\tridots$. Consider the following obstacle $\gamma_n:\SG_n\to\R$ in the rotor aggregation model, similar to \eqref{eq:obst-div-sand} for the divisible sandpile model
\begin{equation}\label{eq:obst-rotor}
	\gamma_n(x)=-\greenOperatorDiscrete^n (S_{\kappa(n)}\sigma_n(x)-1)
\end{equation}
and let $s_n$ be the least superharmonic majorant of $\gamma_n$. Then the difference $s_n-\gamma_n$ is the odometer function for the divisible sandpile started from initial density $S_{\kappa(n)}\sigma_n$ on $\SG_n$. In what follows we compare the smoothed odometer for rotor aggregation started from $\sigma_n$ with the odometer for the divisible sandpile started from initial configuration $S_{\kappa(n)}\sigma_n$.

\begin{lem}\label{lem:smoothed_odometer_comparison_to_DS}
	Let $\domTriangle=B(0,2^L)\subset \SG$ be a double triangle centered at the origin that contains the divisible sandpile clusters started from initial density $S_{\kappa(n)}\sigma_n$ on $\SG_n$. Then there exists a constant $m_\domTriangle$ depending on $\domTriangle$ (or $L$ resp.) such that for any $n\in\N$
	\begin{align*}
	    S_{\kappa(n)}u_n\geq s_n-\gamma_n-C_0 m_\domTriangle \kappa(n)^{-1/4+\delta\alpha/2} \quad \text{on } \domTriangle^\tridots, \quad \text{for } 0<\delta<1/2\alpha.
	\end{align*}
\end{lem}
\begin{proof}
	Let $\Phi_\domTriangle:\domTriangle\rightarrow[0,\infty)$ be the unique function with $\Delta_\domTriangle \Phi_\domTriangle = 3/4$ and $\Phi_{\domTriangle}(0) = 0$. Since $\Phi$ is subharmonic, it attains its maximum in $\partial \domTriangle$.
	One can construct $\Phi_\domTriangle$ as in Appendix \ref{sec:appendixB} by inductively choosing values such that $\Delta_{n} \Phi^\tridots_\domTriangle= 1$ for all $n\in\N$.
	We then set
	\begin{align*}
		m_\domTriangle:= \sup_{x\in \domTriangle} \Phi_\domTriangle(x) = 5^{L},
	\end{align*}
again in view of Appendix \ref{sec:appendixB}.
	On $\domTriangle^\tridots$ we have $\Delta_n \gamma_n=-1+S_{\kappa(n)}\sigma_n$,
	so by Lemma \ref{lem:Laplacian-of-smoothed-odometer}
	\begin{align*}
	    f(x)=S_{\kappa(n)}u_n(x)+\gamma_n(x)+C_0\kappa(n)^{-1/4+\delta\alpha/2}(m_\domTriangle-\Phi_\domTriangle(x)), \quad x\in \domTriangle^\tridots
	\end{align*}
	is superharmonic on $\domTriangle^\tridots$. Since $f\geq\gamma_n$, we have $f \geq s_n$ on $\domTriangle$ and the claim holds on $\domTriangle$.
\end{proof}
\begin{lem}\label{lem:contained_in_ball_rotor}
	For $n\in\N$, let $\RR_n\subset \SG_n$ be the rotor cluster in $\mathsf{SG}_n$ with initial density $\sigma_n:\SG_n\to\N$. Then there exists a double triangle $\domTriangle = B(0,2^L)\subseteq \SG$ such that $\bigcup_{n\geq 0} \RR_n\subseteq\domTriangle$.
\end{lem}
\begin{proof}
 By assumption \eqref{cond:Starting4} there is a double triangle $B = B(0,2^l)\subseteq \mathsf{SG}$ containing the supports of  all $\sigma_n$ for all $n$. Let $\mathcal{A}_n$ be the rotor   cluster in $\mathsf{SG}_n$ resulting from initial density $\tau(x)=M\mathds{1}_{B^\tridots}(x)$, where $M$ is the upper bound of $\sigma,\sigma_n$ from equation \eqref{cond:Starting1}. By the Abelian property of the rotor model  $\RR_n\subseteq \mathcal{A}_n$. By the inner bound \cite[Theorem 1]{limit-shape-rotor-div-sandpile}, if we start $\lfloor5\delta_n^{-\alpha}\mu(B)\rfloor \geq |B^\tridots|$  particles at the origin in $\mathsf{SG}_n$, then the resulting cluster contains $B^\tridots$. So once again, in view of the Abelian property, if one starts $M\lfloor5\delta_n^{-\alpha}\mu(B\cap \SG_n)\rfloor$  particles at the origin, the resulting set $\domTriangle_n$ of fully occupied sites contains $\mathcal{A}_n$. Together with the outer bound \cite[Theorem 1]{limit-shape-rotor-div-sandpile}, $\domTriangle_n$ is contained in a double triangle in $\SG$ of volume $6M\mu(B)$. Now choose $L$ large enough to deduce the result. The last factor 6 comes from the fact that the outer bound is slightly bigger than the initial radius (radius + 1). 
\end{proof}
For $k\geq 0$, define
\begin{align*}
    \RR_n^k=\{x\in \RR_n:\  B^n(x,k) \subseteq \RR_n\}.
\end{align*}
 Similar to \cite[Lemma 2.7]{LP-mult-idla}, we  investigate the growth of the odometer function in rotor aggregation near the boundary of the rotor cluster $\RR_n$.
\begin{lem}\label{lem:growth_odometer_rotor}
	Let $k\geq 0$. There exists a constant $c>0$ such that for all $x\in \SG_n \backslash\RR_n^{k}$
	\begin{align*}
	    u_n(x)\leq c(M+6)k^\beta\delta_n^\beta.
	\end{align*}
\end{lem}
\begin{proof}
	Let $x\notin \RR_n^{k}$.
	The odometer  $u_n$ defined in \eqref{cond:RROdometerConvergence2} has Laplacian bounded by $M+6$ due to (\ref{eq:RR-bound-odometer-Laplacian})  and by the definition of $\RR_n^k$ we know that there exists $y\in B^n(x,k)$ with $u_n(y)=0$. 
	Define the functions $f_{\pm}:\SG_n\to\R$  as
	\begin{align*}
	    f_{\pm}(z)=u_n(z)\pm  (M+6) \delta_n^\beta \phi_{y,2k}(z),
	\end{align*}
	where  $\phi_{y,2k}(z)$ is as in \eqref{eqref:phi-expec}. 
	It is easy to check that $\Delta_{\mathsf{SG}_n} \phi_{y,2k}= 1$  and  $\Delta_n \phi_{y,2k} = \delta_n^{-\beta}$ on $B^n(y,2k)$. Therefore $f_+ \geq u_n$ is subharmonic, while $f_- \leq u_n$ is superharmonic. Let $B^n=B^n(y,2 k)$ and $h_\pm$ be functions defined on $B^n$ agreeing with $f_\pm$ on $\partial B^n$ and harmonic on $B^n\backslash \partial B^n$. Then by construction and \eqref{property:exit_time} it holds
	\begin{align*}
	    f_+ \leq h_+ &\leq h_- + 2(M+6) \delta_n^\beta \max_{z\in\partial B^n} \phi_{y,2k}(z) \\
	    &\leq h_- + 2 (M+6) \delta_n^\beta c (2k)^\beta 
	    \leq f_-  + 2 (M+6) \delta_n^\beta c (2k)^\beta,
	\end{align*}
	which implies, for  $c' = c 2^{\beta+1}$
	$$ h_+(y) \leq f_-(y) + c' (M+6) \delta_n^\beta k^\beta = c' (M+6) \delta_n^\beta k^\beta.$$
	Equation \eqref{property:EHI}  implies the existence of $c''>0$ with
	$ h_+(z) \leq c'' h_+(y)$ for all $z\in B^n(y, k)$:
	$$ f(z) \leq f_+(z) \leq h_+(z) \leq c'' h_+(y) \leq c'' c' (M+6) \delta_n^\beta k^\beta.$$
\end{proof}
\begin{proof}[Proof of Theorem \ref{thm:conv_odometer_rotor}]
	By Lemma \ref{lem:contained_in_ball_rotor} there exists a double triangle $\domTriangle\subset\SG$
	 containing the cluster of occupied sites for the rotor aggregation started from $\sigma_n$ and the divisible sandpile cluster started from initial density $S_{\kappa(n)}\sigma_n$ for all $n$. We  enlarge $\domTriangle$ to contain the support of $S_{\kappa(n)}u_n$ for all $n$, and we define the function
	\begin{align*}
	    \psi(x)=S_{\kappa(n)}u_n(x)-s_n(x)+\gamma_n(x)+C_0\kappa(n)^{-1/4+\delta\alpha/2}\xi_n(x),
	\end{align*}
	which in view of Lemma \ref{lem:Laplacian-of-smoothed-odometer} is subharmonic for $0<\delta<1/2\alpha$ on the set $\RR_n^{\kappa(n)}$. Above $\xi_n$ is the function with constant Laplacian $1$ on $\domTriangle$; see Appendix \ref{sec:appendixB} for details on how to calculate such a function. For  $y\notin \RR_n^{\kappa(n)}$, by Lemma \ref{lem:growth_odometer_rotor} we have $u_n(y)\leq c(M+6)\kappa(n)^\beta\delta_n^\beta$, which together with
	 Lemma \ref{lem:smoothing_odometer_error_bound} gives
	\begin{align*}
	    S_{\kappa(n)}u_n(y)\leq \delta_n^\beta\left(c(M+6)\kappa(n)^\beta+(M+6)\kappa(n)/2+C_0\kappa(n)^{3/4+\delta\alpha/2}\right).
	\end{align*}
	The RHS above is at most $C_1\kappa(n)^\beta\delta_n^\beta$ for some $C_1>0$. Since $s_n\geq\gamma_n$, for all $y\in\domTriangle\backslash \RR_n^{\kappa(n)}$ 
	\begin{align*}
	    \psi(y)\leq C_1\kappa(n)^\beta\delta_n^\beta+C_0\kappa(n)^{-1/4+\delta\alpha/2}r,
	\end{align*}
	where $r$ is a uniform bound for all $\xi_n$ on $\domTriangle$, $n\in\N$. By the maximum principle, this bound holds on $\domTriangle$. Using Lemma \ref{lem:smoothed_odometer_comparison_to_DS} we finally get
	\begin{align*}
	    -C_0r\kappa(n)^{-1/4+\delta\alpha/2}\leq S_{\kappa(n)}u_n-s_n+\gamma_n
	    \leq C_1\kappa(n)^\beta\delta_n^\beta+C_0r\kappa(n)^{-1/4+\delta\alpha/2}
	\end{align*}
	on  $\domTriangle\subset\SG$. By Corollary \ref{cor:obstacleConvergence} and Lemma \ref{lem:convergenceSuperharmonicMajorant}, $\gamma_n$ and $s_n$ converge uniformly to $\gamma$ and $s$ on $\domTriangle$ as $n\to\infty$. Since $\kappa(n)\rightarrow\infty$ and $\kappa(n)\delta_n\rightarrow 0$ as $n\to\infty$, we can conclude that $S_{\kappa(n)}u_n$ converges uniformly to $s-\gamma$ on $\domTriangle$. Because both $S_{\kappa(n)}u_n$ and the odometer function for the divisible sandpile vanish outside $\domTriangle$, this implies that $S_{\kappa(n)}u_n$ has the same limit as the odometer functions of the divisible sandpile on $\mathsf{SG}_n$. Lemma \ref{lem:smoothing_odometer_error_bound} implies then that the odometer functions $u_n$ for rotor aggregation as defined in \eqref{cond:RROdometerConvergence2} converge uniformly to the same odometer $u$ of the divisible sandpile on $\SG$ with initial mass  $\sigma$.
\end{proof}

Once we have shown the uniform convergence of the odometer functions for rotor aggregation on $\SG_n$ to the odometer function $u$ of the divisible sandpile on $ \SG$, it remains to show that the same holds for the corresponding domains of occupied sites by particles doing rotor walks.

\subsection{Convergence of domains in rotor aggregation}\label{sec:RR-Domains}

In addition to the assumptions from Section \ref{sec:RR-odometer}, we require for $\sigma:\SG\to\R$ that for all $x\in \mathsf{SG}$ either $\sigma(x)\geq1$ or $\sigma(x)=0$ and
\begin{align}
    \{\sigma\geq1\}=\overline{\{\sigma\geq 1\}^\circ}.\label{cond:RRconditionOdometer1}
\end{align}
Moreover, we assume that for any $\varepsilon>0$ there exists $N(\varepsilon)$ such that
\begin{align}
     \text{if }x\in\{\sigma\geq 1\}_\varepsilon,\text{ then }\sigma_n(x)\geq1\text{ for all }n\geq N(\varepsilon),\label{cond:RRconditionOdometer2}
\end{align}
and
\begin{align}
    \text{if } x\notin\{\sigma\geq1\}^\varepsilon,\text{ then }\sigma_n(x)=0\text{ for all }n\geq N(\varepsilon).\label{cond:RRconditionOdometer3}
\end{align}
\begin{thm}\label{thm:conv_domain_rotor}
	Let $\sigma$ and $\sigma_n$ fulfill \eqref{cond:RROdometerConvergence1} and \eqref{cond:RRconditionOdometer1}-\eqref{cond:RRconditionOdometer3} as well as \eqref{cond:Starting1}, \eqref{cond:Starting2} and \eqref{cond:Starting4}. For $n\geq1$, let $\RR_n$ be the cluster of occupied sites by rotor particles started from initial particle configuration $\sigma_n:\mathsf{SG}_n\to\N$. For any $\varepsilon>0$ and $n$ large enough we have
	\begin{align*}
	   \widetilde{D}_\varepsilon\cap\SG_n\subseteq \RR_n\subseteq  \widetilde{D}^\varepsilon,
	\end{align*}
	where $D$ is the noncoincidence set on $\mathsf{SG}$ started from $\sigma:\SG\to\R$
	and $ \widetilde{D}=D\cup\{\sigma\geq 1\}^{\circ}.$
\end{thm}
For the rest of this section, we fix again a ball $\domTriangle=B(0,2^l)\subset\SG$ of radius $2^l$ big enough, which contains all rotor clusters $\RR_n\subset\SG_n$ started with particle configurations $\sigma_n:\SG_n\to\N$, for all $n$, as well as all the divisible sandpile clusters $\DD_n\subset\SG_n$ started from initial density $S_{\kappa(n)}\sigma_n$ and its limit shape. Then on $\domTriangle\subset\SG$ we can write
\begin{align*}
    \tilde{D} =\{s>\gamma\}\cup\{\sigma\geq 1\}^\circ,
\end{align*}
where $\gamma:\SG\to\R$ is the obstacle associated with $\sigma$ on $\domTriangle$ as in Definition \ref{def:contSandpile} and $s$ its superharmonic majorant. We also assume that $\domTriangle$ is chosen large enough such that it contains any of the inflations of sets used in the upcoming proofs.

\begin{lem}\label{lem:points_in_conf_around_ball_rotor} 
		Take $x\notin  \widetilde{D}^{\varepsilon/2}$, and fix $\varepsilon>0$ and $n\geq N(\varepsilon/4)$. Denote by $\mathsf{N}_r(x)$ the number of sites in $\RR_n$ which are contained in the ball $B^n(x,r)$ of radius $r$ around $x$ in $\SG_n$ for $1\leq r \leq \varepsilon/(4\delta_n)$. If $m\geq \delta_n^{-\beta}u_n(x)$ for every $x\in B^n(x,r)$, then it holds:
		\begin{align*}
			\mathsf{N}_r(x)\geq (1+1/m)\mathsf{N}_{r-1}(x).
		\end{align*}
	\end{lem}
	\begin{proof}
		It is first easy to deduce $\{\sigma\geq 1\}^{\varepsilon/4}\subseteq  \widetilde{D}^{\varepsilon/4}$ from  $\{\sigma\geq1\}=\overline{\{\sigma\geq1\}^\circ}$. Also $B^n(x,r)\cap  \widetilde{D}^{\varepsilon/4}=\emptyset$ which results in $\sigma_n(y)=0$ for all $y\in B^n(x,r)$. By the definition of $\mathsf{N}_r(x)$, at least $\mathsf{N}_{r-1}(x)$ particles have to enter $B^n(x,r-1)$, which implies that for every $n\in\N$:
		\begin{align*}
			\sum_{y\in\partial B^n(x,r)}u_n(y)\geq \delta_n^\beta\mathsf{N}_{r-1}(x).
		\end{align*}
		There are at most $\mathsf{N}_r(x)-\mathsf{N}_{r-1}(x)$ nonzero terms in the LHS of the sum above, each of them being not bigger than $\delta_n^\beta m$, thus
		\begin{align*}
			m(\mathsf{N}_r(x)-\mathsf{N}_{r-1}(x))\geq \mathsf{N}_{r-1}(x).
		\end{align*}
\end{proof}
\begin{lem}\label{lem:lower_bound_odometer_rotor}
		Let $\varepsilon>0$, then for $n$  large enough and any $x\in \RR_n$ with $x\notin \widetilde{D}^{\varepsilon/2}$ there exists $y\in B^n(x,\varepsilon/(4\delta_n))\cap \RR_n$ such that
		\begin{align*}
			u_n(y)\geq\frac{\varepsilon\delta_n^{\beta-1}}{16\log(C(\frac{\varepsilon}{4\delta_n})^\alpha)}.
		\end{align*}
	\end{lem}
	\begin{proof}
Let $n\geq N(\varepsilon/4)$ be large enough and $B^n=  B^n(x,\varepsilon/(4\delta_n))$ giving $\sigma_n(z) = 0$ for all $z\in B^n$. Set $m = \delta_n^{-\beta}\max_{z\in B^n} u_n(z)$ and notice that $\mathsf{N}_0(z)=\mathds{1}_{\RR_n}(z)$ for all $z\in \SG_n$. Due to Lemma \ref{lem:points_in_conf_around_ball_rotor}, we can bound
		\begin{align*}
			\mathsf{N}_{\varepsilon/(4\delta_n)}(x)\geq (1+1/m)^{\lfloor\varepsilon/4\delta_n\rfloor}\mathsf{N}_0(x)\geq\exp\Big(\frac{\varepsilon}{16m\delta_n}\Big).
		\end{align*}
		Finally by (\ref{property:volume_growth}) we have $\mathsf{N}_{\varepsilon/(4\delta_n)}(x)\leq \big|B^n \big|\leq C(\frac{\varepsilon}{4\delta_n})^\alpha$, 
		which yields 
		$\frac{\varepsilon}{16m\delta_n}\leq\log\Big(C\Big(\frac{\varepsilon}{4\delta_n}\Big)^\alpha\Big)$.
\end{proof}	
The next result tells us that, whenever $x\in \RR_n$ but $x\notin \tilde{D}^\varepsilon$ with large $S_k u_n(x)$ there has to be a $y$ close to $x$ with even larger $S_k u_n (y)$.
	\begin{lem}\label{lem:point_close_bigger_rotor}
		Let $C_0$ be the constant from Lemma \ref{lem:Laplacian-of-smoothed-odometer} and  $0<\delta<1/2\alpha$ and  let $\varepsilon>0$ and $k$ large enough such that $C_0k^{-1/4+\delta\alpha/2}$<1/2. There exist $c, C_1>0$ such that for large $n$ and any $x\in \RR_n \backslash \widetilde{D}^{3\varepsilon/4}$ with $S_ku_n(x)>C_1k^\beta\delta_n^\beta$, there exists $y\in B^n(x,\varepsilon/(2\delta_n)+1)$ with
		\begin{align*}
			S_ku_n(y)\geq S_ku_n(x)+c\varepsilon^\beta.
		\end{align*}
	\end{lem}
	\begin{proof}
		Lemma \ref{lem:smoothing_odometer_error_bound} gives
		\begin{align*}
			u_n(x)\geq S_ku_n(x)-\delta_n^\beta(k(M+6)/2+C_0k^{3/4+\delta\alpha/2})>c(M+6)k^\beta\delta_n^\beta,
		\end{align*}
		for the choice $C_1=c(M+6)+(M+6)/2+C_0$. Hence  $x \in \RR_n^{k}$ by Lemma \ref{lem:growth_odometer_rotor}, and together with Lemma \ref{lem:Laplacian-of-smoothed-odometer}, we obtain for any $y\in \RR_n^{k}$
		\begin{align*}
			\Delta_n S_ku_n(y)\geq 1-S_k\sigma_n(y)-C_0 k^{-1/4+\delta\alpha/2}.
		\end{align*}
		By our choice of $k$ we have $C_0 k^{-1/4+\delta\alpha/2}$<1/2. Taking $n$ large enough such that $k\delta_n<\varepsilon/8$ and $n>N(\varepsilon/8)$, $S_k\sigma_n$ vanishes on $B^n(x,\varepsilon/(2\delta_n))$ and therefore 
		\begin{align*}
			f(z)=S_ku_n(z)+\frac{\delta_n^\beta}{2}\mathbb{E}_z\Big[\tau^n_{\varepsilon/(2\delta_n)}(x)\Big]
		\end{align*}
		is subharmonic on $\RR_n^{k}\cap B^n(x,\varepsilon/(2\delta_n))$ and attains its maximum on the boundary. For $z\in\partial \RR_n^{k}$, Lemma \ref{lem:growth_odometer_rotor} gives $u_n(z)\leq c(M+6)k^\beta\delta_n^\beta$. But $u_n(x)> c(M+6)k^\beta\delta_n^\beta$ so the maximum cannot be attained on $\partial \RR_n^{k}$ and instead has to be attained in $y\in\partial B^n(x,\varepsilon/(2\delta_n))$ which implies the desired inequality since $\mathbb{E}_y[\tau^n_{\varepsilon/(2\delta_n)}(x)] \leq c (\varepsilon/\delta_n)^\beta$ by \eqref{property:exit_time} .
	\end{proof} 
Now we can finally prove the scaling limit for the rotor aggregation cluster $\RR_n$.
\begin{proof}[Proof of Theorem \ref{thm:conv_domain_rotor}]
		Take $\varepsilon>0$ and recall that $u=s-\gamma$, $D=\{u>0\}$ and  finally $ \widetilde{D}=D\cup\{\sigma\geq 1\}^{\circ}$. Choose $\eta>0$ such that
$\widetilde{D}_\varepsilon\subseteq D_\eta\cup\{\sigma\geq 1\}_\eta$ and
 $m_\eta>0$ such that $u\geq m_\eta$ on $D_\eta$, which is possible since $\overline{D_\eta} \subsetneq D$. By Theorem \ref{thm:conv_odometer_rotor}, we get for $n$ large enough 
\begin{align*}
u_n(x)>u(x)-m_\eta/2>0 \text{ for all } x\in D_\eta\cap \mathsf{SG}_n.
\end{align*}
On the other hand, $\{\sigma\geq 1\}_\eta\cap \mathsf{SG}_n\subseteq \RR_n$ for large $n$ by \eqref{cond:RRconditionOdometer2}, which concludes the lower bound.
For the upper bound, we fix $x\in \mathsf{SG}_n$ with $x\notin \widetilde{D}^\varepsilon$. As an immediate consequence, $u=0$ on $B(x,\varepsilon)\subset\SG$. Again, by Theorem \ref{thm:conv_odometer_rotor} it holds for large {enough $n$}
	\begin{align*}
		u_n(z)<c\varepsilon^\beta/2 \text{ for all } z\in B(x,\varepsilon)\cap\mathsf{SG}_n, 
	\end{align*} where $c$ is the constant from Lemma \ref{lem:point_close_bigger_rotor}. Let $k\geq (2C_0)^{1/(1/4-\delta\alpha/2)}$ as in Lemma \ref{lem:point_close_bigger_rotor}, where $C_0$ is the constant from Lemma \ref{lem:Laplacian-of-smoothed-odometer}. Taking $n$ large enough such that $ |S_ku_n-u_n|\leq c\varepsilon^\beta/2$, which is possible in view of Lemma \ref{lem:smoothing_odometer_error_bound}, yields $S_ku_n<c\varepsilon^\beta$ on $B^n(x,\varepsilon/\delta_n)\cap\SG_n$. Lemma \ref{lem:point_close_bigger_rotor} shows that $S_ku_n\leq C_1k^\beta\delta_n^\beta$ on $B^n(x,\varepsilon/(4\delta_n))$ or $x\notin \mathcal{R}_n$. In the former case for some constant $C_2>0$
	\begin{align*}
		u_n(y)\leq \frac{S_ku_n(y)}{\mathbb{P}_y[Y^n_k=y]}<C_2k^{\beta+\alpha/\beta}\delta_n^\beta \text{ for all } y\in B^n(x,\varepsilon/(4\delta_n)),
	\end{align*} 
	where the first inequality holds in view of the definition of $S_k$ and the last inequality is due to Lemma \ref{lem:BoundTransitionProbabilities}\textit{(3)}.
	The RHS decreases faster than $\varepsilon\delta_n^{\beta-1}/(16\log(C\varepsilon^\alpha/(4\delta_n)^\alpha))$, so by Lemma \ref{lem:lower_bound_odometer_rotor}, $x\notin \RR_n$ and therefore $\RR_n\subset \tilde{D}^\varepsilon$.
\end{proof}

\section{Scaling limit for internal DLA}\label{sec:idla}

We investigate here internal DLA. While the proofs follow mostly the lines of \cite{LP-mult-idla} and \cite{idla_cati} and are based on the ideas introduced in \cite{idla_lawler}, we include some details here since we also have to take care of the fractal structure of the Sierpi\'nski gasket. We follow the usual approach and we split the proof into two parts. For the {\it inner estimate},
we show that for any $\varepsilon>0$ the internal DLA cluster $\II_n$ will eventually contain the $\epsilon$-deflation of the set $D=\{u>0\}$, where $u$ is the limit of the odometer for divisible sandpiles started from our smoothed densities. For the {\it
outer estimate} we show that the $\epsilon$-inflation of $D$ will eventually contain the internal DLA cluster. For the outer estimate it is important that the shells $D\backslash D_\varepsilon$ do not grow too fast, which is the case in $\SG$. 
We prove the following.
\begin{thm}\label{thm:IDLA}
Let $\sigma$ and $\sigma_n$ initial configurations on $\SG$ and $\SG_n$ respectively that fulfill \eqref{cond:RROdometerConvergence1} and \eqref{cond:RRconditionOdometer1}-\eqref{cond:RRconditionOdometer3} as well as \eqref{cond:Starting1}, \eqref{cond:Starting2} and \eqref{cond:Starting4}.  Then, for every $\varepsilon>0$ with probability $1$ we have
\begin{align*}
    \widetilde{D}_\varepsilon \cap \mathsf{SG}_n \subseteq\II_n\subseteq \widetilde{D}^\varepsilon
\end{align*}
for  $n$ large enough, where $\II_n$ is the internal DLA cluster started from configuration $\sigma_n$ on $\SG_n$ and
$\widetilde{D}=\{s>\gamma\}\cup\{\sigma\geq 1\}^\circ$.
\end{thm}

\subsubsection*{Inner estimate}
Since we are using the standard approach from \cite{idla_lawler} and \cite{LP-mult-idla}, we will not motivate the random variables below. We fix $ n\in\N$ and the initial configuration $\sigma_n$ on $\SG_n$. We label the particles that generate $\II_n$ by $1,\ldots,m_n$, with $m_n=\sum_{x\in\SG_n}\sigma_n(x)$ being the total number of particles. Denote by $x_i$ the starting position of particle $i$ such that 
$\#\{i|x_i=x\}=\sigma_n(x)$, and for $i=1,...,m_n$ let $(X_t^n(i))_{t\in\N}$ be a sequence of $m_n$ simple independent random walks in $\mathsf{SG}_n$ starting in $x_i$. For $z\in \mathsf{SG}_n$ and $\varepsilon>0$ consider the stopping times
\begin{align*}
    &\tau_z^i :=\inf\{t\geq 0:\ X_t^n(i)=z\},\\
    &\tau_\varepsilon^i:=\inf\{t\geq 0:\ X_t^n(i)\notin D_\varepsilon\cap \mathsf{SG}_n\}, \\
    &\nu^i:=\inf\{t\geq0:\ X_t^n(i)\notin\{X_{\nu^j}^n(j)\}_{j=1}^{i-1}\},
\end{align*}
and we also set $\nu^1:=0$ and we remark that all three above stopping times depend on the iteration $n$, but we omit $n$ and write only $\tau_z^i$ (resp. $\tau_{\varepsilon}^i$ and $\nu^i$) instead of $\tau_z^i(n)$ (resp. $\tau_{\varepsilon}^i(n)$ and $\nu^i(n)$). If $\II_n^j$ denotes the cluster after $j$ particles settled, then $\II_n^j=\II_{n}^{j-1}\cup\{X^n_{\nu^j}(j)\}$. Fix $z\in D_\varepsilon\cap \mathsf{SG}_n$ and consider the random variables
\begin{align*}
    \mathcal{M}_\varepsilon:=\sum_{i=1}^{m_n}\mathds{1}_{\{\tau_z^i<\tau_\varepsilon^i\}} \quad \text{and} \quad 
    L_\varepsilon:=\sum_{i=1}^{m_n}\mathds{1}_{\{\nu^i\leq \tau_z^i<\tau_\varepsilon^i\}}.
\end{align*}
If $L_\varepsilon<\mathcal{M}_\varepsilon$, then $z\in \II_n$. $L_\varepsilon$ being a sum of non-independent random variables, we use another random variable to upper bound it.
For each $y\in D_\varepsilon\cap \mathsf{SG}_n$, let $(Y_t^n(y))_{t\in \N}$ be a simple random walk on $\mathsf{SG}_n$ starting at $y$, with $Y^n_t(x)$ and $Y^n_t(y)$ independent if $x\neq y$, and
\begin{align*}
\widetilde{L}_\varepsilon:=\sum_{y\in D_\varepsilon\cap \mathsf{SG}_n}\mathds{1}_{\{\tau_z^y<\tau_\varepsilon^y\}}
\end{align*}
where
$\tau_z^y=\inf\{t\geq 0:\ Y_t^n(y)=z\}$ and $\tau_\varepsilon^y=\inf\{t\geq 0:\ Y_t^n(y)\notin D_\varepsilon\cap \mathsf{SG}_n\}$.
Then indeed $L_\varepsilon\leq\widetilde{L}_\varepsilon$, and $\widetilde{L}_\varepsilon$ is a sum of independent indicators. Define now 
\begin{align*}
 \greenFunctionUnscaled^{n,\varepsilon}(y,z):=\mathbb{E}\left[\#\{t<\tau_\varepsilon^y:\ Y_t^n(y)=z\}\right]
 \text{ and }
    f^{n,\varepsilon}(z):=\greenFunctionUnscaled^{n,\varepsilon}(z,z)\mathbb{E}[\mathcal{M}_\varepsilon-\widetilde{L}_\varepsilon].
\end{align*}
Be aware that $\greenFunctionUnscaled^{n,\varepsilon}$ is the unscaled version of the Green function  defined in
(\ref{eq:green-visits}), which is handier for internal DLA. By definition, $f^{n,\varepsilon}$ solves the Dirichlet problem
\begin{align*}
    \Delta_n f^{n,\varepsilon}&=\delta_n^{-\beta}(1-\sigma_n),\text{ on } D_\varepsilon\cap \mathsf{SG}_n ,\\
    f^{n,\varepsilon}&=0,\text{ on } D_\varepsilon^C\cap \mathsf{SG}_n,
\end{align*}
and note that the unscaled divisible sandpile odometer $\delta_n^{-\beta}u_n$ solves the same Dirichlet problem on the set $\DD_n=\{u_n>0\}$.

\begin{lem}\label{lem:odometer_smoothing_error_idla}
	Suppose $\sigma_n$ is bounded by $M$ and has finite support. Let $u_n, \tilde{u}_n$ be the odometer functions of the divisible sandpile with initial densities $\sigma_n, S_k \sigma_n$, respectively. Then
	\begin{align*}
		|u_n-\tilde{u}_n|\leq\frac{1}{2}kM\delta_n^\beta \text{ for all } k\in\N. 
	\end{align*}
\end{lem}
\begin{proof}
In view of $S_1 f=f+\frac{\delta_n^\beta}{2}\Delta_n f$ and $S_{i+l}=S_i S_l$ we obtain by induction over $k$ that\begin{align*}
    S_k\sigma_n=\sigma_n+\Delta_n w_k, \quad \text{with} \quad
    w_k=\frac{\delta_n^\beta}{2}\sum_{j=0}^{k-1}S_j\sigma_n.
\end{align*}
We remark that $w_k$ is non-negative and bounded from above by $\frac{1}{2}Mk\delta_n^\beta$ due to the boundedness of $\sigma_n$ and thus also $S_j\sigma_n$ for all $j\in\N$.
Let us now denote by $\nu_n$ (resp. $\tilde{\nu}_n$) the limit mass configuration on $\mathsf{SG}_n$ for the divisible sandpile started from $\sigma_n$ (resp. $S_k\sigma_n$). The sum $w_k+\tilde{u}_n$, where $\tilde{u}_n$ is the odometer for divisible sandpile started from $S_k\sigma_n$, solves the inequality
\begin{align}\label{eq:sandpile_equation}
\sigma_n+\Delta_n v \leq 1\text{ and }v\geq 0
\end{align}
in view of
\begin{align*}
\tilde{\nu}_n=S_k\sigma_n+\Delta_n \tilde{u}_n=\sigma_n+\Delta_n(\tilde{u}_n+w_k),
\end{align*} 
and the fact that $\tilde{\nu}_n\leq 1$. Thus by the least action principle as introduced in Section \ref{sec:div-sand} (which states that the odometer function is the smallest solution to \eqref{eq:sandpile_equation}) we obtain $u_n\leq \tilde{u}_n+w_k$. In a similar manner we can see
\begin{align*}
    \nu_n&=\sigma_n+\Delta_n u_n
    = S_k\sigma_n+\Delta_n(u_n-w_k+\frac{1}{2}Mk\delta_n^\beta),
\end{align*}
where again $\nu_n\leq 1$ and $u_n-w_k+\frac{1}{2}Mk\delta_n^\beta\geq 0$, thus again by the least action principle we obtain $\tilde{u}_n\leq u_n-w_k+\frac{1}{2}Mk\delta_n^\beta$.
\end{proof}

\begin{lem}\label{lem:conv_sandpile_domain_idla}
	Suppose $\sigma,\sigma_n$ fulfill the assumptions in Theorem \ref{thm:IDLA} and let $u_n$ be the odometer function for the divisible sandpile on $\mathsf{SG}_n$ with initial density $\sigma_n$ and $\DD_n=\{u_n>0\}$. Then $$u_n^\doubleTriangle \rightarrow s-\gamma \text{ uniformly as }n\to\infty.$$
In particular, for all $\varepsilon>0$ it holds $D_\varepsilon\cap \mathsf{SG}_n\subseteq \DD_n$ for $n$ large enough.
\end{lem}
\begin{proof}
For the odometer function $\tilde{u}_n$ of the divisible sandpile with initial density $S_{\kappa(n)}\sigma_n$ as in (\ref{cond:RROdometerConvergence1}), in view of Theorem \ref{theo:convergenceSandpileOdometer} we have $\tilde{u}_n^\doubleTriangle \rightarrow u$ uniformly, as $n\to\infty$. Together with  Lemma \ref{lem:odometer_smoothing_error_idla} this gives $u^\doubleTriangle_n\rightarrow u$ uniformly as $n\to\infty$. Finally, let $b>0$ be the minimum value of $u$ on $\overline{D}_\varepsilon$. For $n$ large enough, we get $u_n>b/2$ on $D_\varepsilon\cap \mathsf{SG}_n$ which yields $D_\varepsilon\cap \mathsf{SG}_n\subseteq \DD_n$.
\end{proof}

\begin{lem}\label{lem:f_lower_bound_idla}
	For any $\varepsilon>0$ let $b=\min\{u(x):\ x\in \overline{D}_\varepsilon\}>0$. There exists $0<\eta<\varepsilon$ such that for all large enough $n$ and all $z \in D_\varepsilon\cap \SG_n$
	\begin{align*}
		f^{n,\eta}(z)\geq\frac{1}{2}b\delta_n^{-\beta}.
	\end{align*}
\end{lem}

\begin{proof}
Let $\eta>0$ be chosen such that $u<b/6$ outside of $D_\eta$, which is possible since $u$ is uniformly continuous on $\overline{D}$. Then $\eta<\varepsilon$ since $u\geq b$ on $\overline{D}_\varepsilon$ by the definition of $b$ as the minimum of $u$ on $\overline{D}_\varepsilon$. Let $u_n$ be the odometer function for the divisible sandpile on $\mathsf{SG}_n$ started from initial density $\sigma_n$ and let $\DD_n$ be the corresponding divisible sandpile cluster.
We can now choose $n_0\in\N$ large enough such that for all $n>n_0$ we have $|u_n-u|<b/6$ on  $D_\eta\cap \SG_n$, and this is possible in view of the uniform convergence from Lemma \ref{lem:conv_sandpile_domain_idla}, hence $u_n\leq b/3$ on $\partial(D_\eta\cap \SG_n)$. Again by Lemma \ref{lem:conv_sandpile_domain_idla} we can adjust $n_0$ such that we also have $D_\eta\cap \SG_n\subseteq \DD_n$ for all $n>n_0$.
We thus obtain that the function $\delta_n^\beta f^{n,\eta}-u_n$ is harmonic on $D_\eta\cap \SG_n$ and it therefore attains its minimum on the outer boundary. However, the function $f^{n,\eta}$ vanishes outside of $D_\eta \cap \SG_n$ by definition, hence we obtain for any $z\in D_\eta\cap \SG_n$
\begin{align*}
\delta_n^\beta f^{n,\eta}(z)\geq u_n(z)-b/3\geq u(z)-b/2.
\end{align*}
By our choice of $b$, we have $u\geq b$ in $D_\varepsilon\cap \SG_n$, and since $\eta<\varepsilon$ the previous inequality gives $\delta_n^\beta f^{n,\eta}\geq b/2$ on $D_\varepsilon\cap \SG_n$.
\end{proof}
Let 
\begin{align*}
    \tilde{\II}_n:=\{X_{\nu^i}^n(i):\ \nu^i<\tilde{\tau}^i\}\subseteq \II_n
\quad \text{with} \quad 
    \tilde{\tau}^i:=\inf\{t\geq 0:\ X_t^n(i)\notin \widetilde{D}\cap \mathsf{SG}_n\}.
\end{align*}
The inner estimate  follows from the following lemma. Although it suffices to prove the inner estimate for $\II_n$ instead of $\tilde{\II}_n$, we use this stronger estimate  in the proof of the outer estimate.
\begin{lem}\label{lem:inner_estimate_idla}
For any $\varepsilon>0$, $\mathbb{P}\big[\widetilde{D}_\varepsilon\cap \mathsf{SG}_n\subseteq \tilde{\II}_n\text{ for all but finitely many }n\big]=1$.
\end{lem}

\begin{proof}
For $z\in \SG_n$, let $E_z(n)$ be the event that $z$ is not contained in the internal DLA cluster on $\SG_n$ with initial density $\sigma_n$. Using condition (\ref{cond:additionalStartingCondition3}), it suffices to prove that for $\varepsilon'>0$ 
\begin{align*}
    \sum_{n\geq1}\sum_{z\in D_{\varepsilon'}\cap \mathsf{SG}_n}\mathbb{P}[E_z(n)]<\infty.
\end{align*}
Let $b := \inf_{z\in D_{\varepsilon'}} u(z) >0 $ and $0 < \eta < \varepsilon'$ as in Lemma \ref{lem:f_lower_bound_idla} such that $ f^{n,\eta}(z)\geq b\delta_n^{-\beta}/2$ for $z \in D_{\varepsilon'}\cap \SG_n$ and sufficiently large $n$. We fix $z\in D_{\varepsilon'}\cap \mathsf{SG}_n$ and choose $R>0$ such that $D\subseteq B(z,R)$ giving 
\begin{align*}
    B(z,\varepsilon'-\eta)\subseteq D_\eta\subseteq B(z,R).
\end{align*} 
Notice that $R$ depends on the size of $\supp(\sigma)$ and the supremum over all values of $\sigma$. Since $L_\eta\leq \tilde{L}_\eta$ we have
\begin{align}
    \mathbb{P}[E_z(n)]&\leq\mathbb{P}[\mathcal{M}_\eta=L_\eta]
    \leq\mathbb{P}[\mathcal{M}_\eta\leq \tilde{L}_\eta]
    \leq \mathbb{P}[\mathcal{M}_\eta\leq a]+\mathbb{P}[\tilde{L}_\eta\geq a] \label{eq:proof:IDLAinner-Prob}
\end{align}
for any $a\in\mathbb{R}$. We need an $a$ such that both summands are small. By \cite[Lemma 4]{idla_lawler}
\begin{align}
    &\mathbb{P}\Big[\tilde{L}_\eta\geq \mathbb{E}[\tilde{L}_\eta] + \mathbb{E}[\tilde{L}_\eta]^{1/2+\gamma}\Big]<2 \exp\big(- \frac{1}{4} \mathbb{E}[\tilde{L}_\eta]^{2\gamma}\big), \label{eq:proof:proofIDLAinner-ProbBound1}\\
    &\mathbb{P}\Big[\mathcal{M}_\eta \leq \mathbb{E}[\mathcal{M}_\eta] - \mathbb{E}[\mathcal{M}_\eta]^{1/2+\gamma}\Big]<2 \exp\big(- \frac{1}{4} \mathbb{E}[M_\eta]^{2\gamma}\big),\label{eq:proof:proofIDLAinner-ProbBound2}
\end{align}
for any $0<\gamma<1/2$. So we find an appropriate $a$ if the interval $$I:=\big[\mathbb{E}[\tilde{L}_\eta] + \mathbb{E}[\tilde{L}_\eta]^{1/2+\gamma}, \mathbb{E}[\mathcal{M}_\eta] - \mathbb{E}[\mathcal{M}_\eta]^{1/2+\gamma}\big] \text{ is non-empty.}$$  
For this, note the lower bound
\begin{align*}
    \mathbb{E}[\mathcal{M}_\eta]-\mathbb{E}[\tilde{L}_\eta]=\frac{f^{n,\eta}(z)}{\greenFunctionUnscaled^{n,\eta}(z,z)}\geq\frac{b}{2\delta_n^\beta \greenFunctionUnscaled^{n,\eta}(z,z)},
\end{align*}
which immediately gives $\E[ M_\eta] \geq \E[ \tilde{L}_\eta]$. So, by choosing $\gamma,\gamma'>0$ such that 
\begin{align*}
	\E [M_\eta] ^{1/2+\gamma} &:= c\Big(\frac{b}{ R^{\beta-\alpha} \delta_n^\alpha}\Big)^{3/4} \leq \frac{b}{4 \delta_n^\beta \greenFunctionUnscaled^{n,\eta}(z,z)},\\
	\E [\tilde{L}_\eta] ^{1/2+\gamma'} &:= c\Big(\frac{b}{ R^{\beta-\alpha} \delta_n^\alpha}\Big)^{3/4} \leq \frac{b}{4 \delta_n^\beta \greenFunctionUnscaled^{n,\eta}(z,z)},
\end{align*}
the interval $I$ is non-empty, where the last inequality is due to \cite[Lemma 2.10]{idla_cati} which states that $\greenFunctionUnscaled^{n,\eta}(z,z) \leq C (R/\delta_n)^{\beta-\alpha}$.
It is not immediately clear that with this choice the condition on $\gamma, \gamma' < 1/2$ is satisfied. 
Standard theory for Markov chains gives 
\begin{align}\label{eq:id_for_L_eta}
\mathbb{E}[\tilde{L}_\eta]=\sum_{y\in D_\eta\cap \SG_n}\mathbb{P}[\tau_z^y < \tau_\eta^y]=\frac{1}{g^{n,\eta}(z,z)}\sum_{y\in D_\eta\cap \SG_n}g^{n,\eta}(z,y)=\frac{\mathbb{E}[\tau_\eta^z]}{g^{n,\eta}(z,z)}.
\end{align}
Applying the lower bounds for expected exit times \eqref{property:exit_time} and the upper bound for the stopped Green function \cite[Lemma 2.10]{idla_cati} yields that $\gamma,\gamma' < 1/2$ since
\begin{align*}
    \mathbb{E}[\tilde{L}_\eta]&=\frac{\mathbb{E}_z\big[ \tau_{D_\eta\cap \mathsf{SG}_n}\big]}{\greenFunctionUnscaled^{n,\eta}(z,z)}
    \geq C'\Big(\frac{\varepsilon'-\eta}{\delta_n}\Big)^\beta\frac{1}{(R/\delta_n)^{\beta-\alpha}}
    = C' \frac{(\varepsilon' - \eta)^\beta}{R^{\beta-\alpha}} \delta_n^{-\alpha}.
\end{align*}
Ultimately our choice of $\gamma$ and $\gamma'$ also gives that the probabilities in (\ref{eq:proof:proofIDLAinner-ProbBound1}) and (\ref{eq:proof:proofIDLAinner-ProbBound2}) are exponentially decreasing, because the random variable $M_\eta$ is bounded by $M\sum_{y\in D_\eta\cap \SG_n}\mathbb{P}[\tau_z^y < \tau_\eta^y]$ which gives $\mathbb{E}[\mathcal{M}_\eta]\leq MR^\beta/\delta_n^\beta \greenFunctionUnscaled^{n,\eta}(z,z)\leq M C R^\alpha \delta_n^{-\alpha}$. Therefore we have
\begin{align*}
	\mathbb{E}[\tilde{L}_\eta]^{2\gamma'} = c\Big(\frac{b}{ R^{\beta-\alpha} \delta_n^\alpha}\Big)^{6/4} \mathbb{E}[\tilde{L}_\eta]^{-1} \geq c' M \Big(\frac{b}{ R^{\beta-\alpha}}\Big)^{6/4} R^{-\alpha} \delta_n^{-\alpha/2}.
\end{align*}
Finally, we obtain 
\begin{align*}
    \sum_{n\geq1}\sum_{z\in D_{\varepsilon'}\cap \mathsf{SG}_n}\mathbb{P}[E_z(n)]\leq \sum_{n\geq 1}C'' R^\alpha \delta_n^{-\alpha}\cdot 4\exp\Big(-c' M b^{3/2} R^\frac{\alpha-\beta}{2} \delta_n^{-\alpha/2} \Big)<\infty,
\end{align*}
which together with the first Borel-Cantelli lemma proves the claim.
\end{proof}
\subsubsection*{Outer estimate}
For the outer estimate we rely on \cite{idla_cati}.
The idea is to argue based on the inner estimate and the fact that the boundary of $D$ is a null set, that not too many particles can leave  $D$. We will then show that the remaining particles will not be able to leave the $\varepsilon$-inflation
$D^{\varepsilon}$ of $D$.
For the remainder we use the notation for deflation and inflation for sets $A\subseteq \mathsf{SG}_n$:
\begin{align*}
    A_{r} = \{ x\in A :  B^n(x,r) \subseteq A\} \quad \text{and}\quad  A^{r} = \{ x\in \mathsf{SG}_n:\  A \cap B^n(x,r) \neq \emptyset \}.
\end{align*}
Recall \cite[Lemma 3.5, Lemma 3.6]{idla_cati}. We first prove the following.
	\begin{lem}\label{lem:exit_inflation_idla}
	Let $n\in \N$, then there exist $\rho,\eta \in (0,1]$ such that for any $D\subseteq \SG_n$ and $r>0$ the following holds. For all $S\subseteq D^r \backslash D$ with $|S|\leq \rho r^\alpha$ and all $x\in D$ we have 
	$$\mathbb{P}_x\big[\tau^n_{S\cup D} < \tau^n_{D^r}\big] \geq \eta.$$
	\end{lem}
	\begin{proof}
	Let us denote $Y := \partial D^{r/2}$ the set of all vertices of distance $\lfloor r/2 \rfloor +1$ from $D$. By the Markov property it is sufficient to prove the statement for starting points $y\in Y$, since every trajectory of a random walk started inside $D$ hits the set $Y$.
	Let $A=B^n(y,r/3)\backslash S$, then, in view of (\ref{property:volume_growth}) with corresponding constant $C\geq 1$, by setting $\rho=4^{-\alpha}/C\in(0,1]$ we obtain
	\begin{align*}
	    |A|\geq |B^n(y,r/3)|-|S|\geq \frac{1}{C}\Big(\frac{r}{3}\Big)^\alpha-\rho r^\alpha\geq\frac{1}{C^2}\Big(1-\Big(\frac{3}{4}\Big)^\alpha\Big)|B^n(y,r/3)|.
	\end{align*}
	The choice $\varepsilon = \frac{1}{C^2}(1-(\frac{3}{4})^\alpha)$ \cite[Lemma 3.5]{idla_cati} gives the existence of $\eta>0$ with $\mathbb{P}_y[\tau^n_A<\tau^n_{r/3}(y)]\geq\eta$.
	\end{proof}									

Next, for $\rho$  as in Lemma \ref{lem:exit_inflation_idla}  we estimate the number of particles leaving a given set during the internal DLA process.
\begin{lem}\label{lem:binomial_bound_idla}
For a bounded set $D\subset \SG_n$ and $r\in\N$, let $N$ be the number of particles that leave $D^r$ during the internal DLA process when starting with all sites in $D$ occupied and $k$ additional particles that start at sites $y_1,...,y_k \in D$. If $k\leq \rho r^\alpha$ where $\rho$ is as in Lemma \ref{lem:exit_inflation_idla}, then $N\leq \text{Binom}(k,p)$ for some constant $p<1$.
\end{lem}

\begin{proof}
Label the additional $k$ particles with $1,..,k$. By $\II^{j}_n$ we denote the cluster after the first $j$ particles have been settled. By the inductive way of building $\II^{j}_n$ it holds $|\II^{j}_n\backslash D|\leq j\leq k\leq \rho r^\alpha$. Denoting by $A_j$ the event that the $j$-th particle ever leaves $D^r$, in view of Lemma \ref{lem:exit_inflation_idla} it holds $\mathbb{P}[A_j|\II^{j-1}_n]\leq 1-\eta$,
for some $\eta\in(0,1)$. For $p=1-\eta$, by coupling the indicators $\mathds{1}_{A_j}$ with  i.i.d indicators $I_j\geq \mathds{1}_{A_j}$ of mean $p$ yields
$    N=\sum_{j=1}^k\mathds{1}_{A_j}\leq\sum_{j=1}^k I_j=\text{Binom}(k,p)$.
\end{proof}
The next lemma shows that if too few particles start inside the given set $D$, then only very few particles can ever reach the outside of $D^{2r}$.
\begin{lem}\label{lem:few_particles_reach_outside_idla}
For $n\in\N$, let $D$ be a bounded subset of $\mathsf{SG}_n$ and let $k,r\in \N$  with
$ k\leq \rho(1-p^{1/2\alpha})^\alpha r^\alpha$,
where $\rho, p<1$ as in Lemma \ref{lem:binomial_bound_idla}. Let $N$ be the number of particles that ever reach the outside of $D^{2r}$ when starting with all sites inside $D$ already occupied and $k$ additional particles that start at sites $y_1,...,y_k \in D$. Then, for constant $c_0,c_1>0$ we have
\begin{align*}
    \mathbb{P}[N>0]\leq c_0e^{-c_1 r}.
\end{align*}
\end{lem}
	\begin{proof}
	For $j\in\N$ let $ r_j:=(1-p^{j/2\alpha})r$, $N_j$ the number of particles that ever leave $D^{r_j}$ and $k_j:=p^{j/2}k$. Taking $h_j:=r_{j+1}-r_{j}$, we have
	\begin{align*}
	    k_j& \leq p^{j/2}\rho(1-p^{1/2\alpha})^\alpha r^\alpha
	    =\rho h_j^\alpha.
	\end{align*}
	We now apply Lemma \ref{lem:binomial_bound_idla} which conditioned on $N_j$ implies
	$N_{j+1} \leq \text{Binom}(N_j,p)$. Letting $A_j:=\{N_j\leq k_j\}$ yields
	\begin{align*}
	    \mathbb{P}[A_{j+1}|A_j]&\geq\mathbb{P}[\text{Binom}(k_j,p)\leq k_{j+1}]
	    \geq 1-2e^{-ck_j},
	\end{align*}
	where the last  inequality follows from \cite[Lemma 5.8]{LP-mult-idla}. Let $J = \lfloor 2{(\log(k)-\log(r))}/{\log(1/p)} \rfloor$ be the minimal	 integer such that $p^{1/2} r \leq k_J < r$. On the event $A_J$, at most $r$ particles ever leave the set $D^r$, hence there are not enough particles left to ever leave $D^{2r}$. Thus $\mathbb{P}[N=0]\geq\mathbb{P}[A_J]$ and 
	\begin{align*}
	    \mathbb{P}[N=0]\geq\mathbb{P}[A_1]\mathbb{P}[A_2|A_1]\ldots \mathbb{P}[A_J|A_{J-1}]\geq 1-2J e^{-c r},
	\end{align*}
	and the right hand side is at least $1-c_0e^{-c_1 r}$ for suitable constants $c_0,c_1>0$.
	\end{proof}
\begin{proof}[Proof of Theorem \ref{thm:IDLA}]
	  Lemma \ref{lem:inner_estimate_idla} gives the inner estimate. For the outer estimate, let $\varepsilon > 0$ and
	\begin{align*}
	    N_n:=\big|\{1\leq i \leq m_n:\ \nu^i\geq \tilde{\tau}^i\}\big|=m_n-\big|\tilde{\II}_n\big|
	\end{align*}
	be the number of particles leaving $\widetilde{D}\cap \mathsf{SG}_n$ before aggregating to the cluster, where $m_n$ is the number of particles initially in $\mathsf{SG}_n$. 
	Let $p<1$ as in Lemma \ref{lem:binomial_bound_idla} and set $K_0:=\rho\left({(1-p^{1/2\alpha})}\right)^\alpha$. Since by Proposition \ref{prop:boundary_conservation_of_mass}, $\mu(\partial\widetilde{D})=0$, for sufficiently small $\eta>0$ we obtain
	\begin{align*}
	    \mu(\widetilde{D}\backslash\widetilde{D}_{2\eta})\leq \frac{1}{2}K_0\varepsilon^\alpha.
	\end{align*}
	Clearly, for large enough $n$ the construction gives
	\begin{align*}
	    \delta_n^\alpha\big|\widetilde{D}_\eta\cap \mathsf{SG}_n\big|=\mu((\widetilde{D}_\eta\cap \mathsf{SG}_n)^{\doubleTriangle})\geq \mu(\widetilde{D}_{2\eta})\geq \mu(\widetilde{D})-\frac{1}{2}K_0\varepsilon^\alpha,
	\end{align*}
	and Proposition \ref{prop:boundary_conservation_of_mass} implies  
	\begin{align*}
	    |\delta_n^\alpha m_n- \mu(\widetilde{D})| < \frac{1}{2}K_0\varepsilon^\alpha \text{ for $n$ large enough},
	\end{align*}
	such that, on the event  $\{\widetilde{D}_\eta\cap \mathsf{SG}_n\subseteq \tilde{\II}_n\}$, we have for $n$ large 
	\begin{align*}
	    N_n&\leq m_n-\big|\widetilde{D}_{\eta}\cap \mathsf{SG}_n\big|
	    \leq m_n-\delta_n^{-\alpha}\mu(\widetilde{D})+\frac{1}{2}K_0\varepsilon^\alpha\delta_n^{-\alpha}
	    \leq K_0 \varepsilon^\alpha\delta_n^{-\alpha}.
    \end{align*}	
    So for the event $A_n = \{ N_n \leq K_0 \varepsilon^\alpha\delta_n^{-\alpha}\}$ by Lemma \ref{lem:inner_estimate_idla} it clearly holds
	\begin{align*}
	    \mathbb{P}\big[\limsup_{n\rightarrow\infty} A_n\big]\geq\mathbb{P}\big[\widetilde{D}_\eta\cap \mathsf{SG}_n\subseteq\tilde{\II}_n\text{ for all but finitely many }\big] = 1.
	\end{align*}
	When considering the probability that $\II_n$ has vertices outside the inflation $ (\widetilde{D}\cap \mathsf{SG}_n)^{\varepsilon/\delta_n}$, conditioned on $A_n$, Lemma \ref{lem:few_particles_reach_outside_idla} then gives
	\begin{align*}
	    \mathbb{P}\big[\{\II_n\cap ((\widetilde{D}\cap \mathsf{SG}_n)^{\varepsilon/\delta_n})^C\neq\emptyset\}\cap A_n\big]\leq c_0e^{-c_1(\varepsilon/\delta_n)},
	\end{align*}
	and the claim follows from Borel-Cantelli lemma.
	\end{proof}
\begin{appendices}
	
	\section{Odometer of the continuous divisible sandpile on $\SG$}\label{sec:appendixB}
	
	
	We give below a recipe on how to calculate the odometer function of the continuous divisible sandpile on $\SG$ as introduced in Section \ref{sec:div-sand}, for a concrete choice of the initial configuration $\sigma:\SG\to[0,\infty)$. More precisely, we take $\sigma$
	to be constant  in the ball of radius $r=2^l$ for some $l\in\N$ and zero outside, that is for $N=2^L> r $ and $L\in \N$ we consider $$\sigma := \frac{\mu(B(0,N))}{\mu(B(0,r))} \mathds{1}_{B(0,r)} = 3^{L-l} \mathds{1}_{B(0,r)},$$
	where we further assume that $L>l$. We take $\domTriangle = B(0,2N)$ and calculate the obstacle $\gamma$, its superharmonic majorant $s$ and show that the corresponding noncoincidence set $D$ for this obstacle is ${D} = B(0,N) \subsetneq \domTriangle$.
	We summarize first the steps we take when calculating the odometer, in order to make the section easier to follow.
	\vspace{-0.25cm}
	\begin{enumerate}
		\setlength\itemsep{0em}
		\item First, we calculate functions $f_i:B(0,2^i)\rightarrow\R$ that have constant Laplacian $1$ i.e.~$\Delta f_i = 1$ and are zero on the boundary i.e.~$f|_{\partial B(0,2^i)}=0$.
		
		\item In the second step, we construct a function $\eta:\domTriangle\rightarrow\R$ with $\Delta \eta = \mathds{1}_{B(0,r)}$. In order to do so, we will take the function $f_l$ (a function with constant Laplacian $1$ on $B(0,r)$ and zero on the boundary) and construct a harmonic function  $h:B(0,2N)\backslash B(0,r)\rightarrow \R$ that matches the normal derivative of $f_l$ on the boundary of $B(0,r)$. By then gluing together $f_l$ and $h$, we obtain the desired function $\eta$.
		
		\item Next, we construct the obstacle $\gamma$ by taking a function $\xi$ with constant Laplacian $1$ on $B(0,2N)$ (we can take $\xi = f_{L+1}$) and setting $\gamma = 3^{L-l}\eta-\xi$.
		
		\item Finally, we show that the function $s$ defined by
		\begin{align*}
			s(x):=\begin{cases}\gamma(x),& x\in B(0,2N)\backslash B(0,N)\\ \gamma(y),& x\in B(0,N)\end{cases},
		\end{align*}
		where $y\in\partial B(0,N)$ and $x\in B(0,2N)$, is superharmonic and dominates $\gamma$. 
	\end{enumerate}
	This shows that for the odometer function $u$ we have $u\leq s-\gamma$. Since $\{s-\gamma>0\}=B(0,N)$ the conservation of mass in Proposition \ref{prop:boundary_conservation_of_mass} yields that $s-\gamma$ has the correct support. Since $s$ is constant on the support of $s-\gamma$, there is no smaller superharmonic function and we have $u=s-\gamma$.
Now we follow the steps in the above summary.
	
	\textbf{Construction of $f_i$.}\label{par:construct_constant_Laplacian}
	To construct $f_i$ we use that $3/4\cdot5^n\Delta_{\mathsf{SG_n}} \rightarrow \Delta$ as $n\to\infty$, in order to calculate $f_i$ on $\mathsf{SG}_n$. The generalized $\frac{1}{5}-\frac{2}{5}$ rule \cite[Theorem A.1]{Div-Sandpile-SG} gives  $f(0) = - {4/3\cdot} 5^i$, and then we can calculate the normal derivatives at the boundary of $B(0,2^i)$. For this, we calculate the values of $f_i$ at positions $a_j,b_j$ as shown in Figure \ref{fig:odometer-construction1}.
	\begin{figure}
		\centering
		\begin{tikzpicture}[scale=0.75]
			\node[right] at (8,0) {$a_0$};
			\node[left] at (0,0) {$b_0$};
			\node[above] at (4,4*1.7320508076) {$0$};
			\node[right] at (6,2*1.7320508076) {\footnotesize $a_1$};
			\node[right] at (5,6/2*1.7320508076) {\footnotesize $a_2$};
			\node[right] at (4.60,3.4*1.7320508076) {\footnotesize.};
			\node[right] at (4.55,3.45*1.7320508076) {\footnotesize.};
			\node[right] at (4.50,3.5*1.7320508076) {\footnotesize.};
			\node[right] at (4.25, 3.75*1.7320508076) {\footnotesize $a_j$};
			
			\begin{scope}[shift={(8,0)}]
				\node[left] at (-6,2*1.7320508076) {\footnotesize $b_1$};
				\node[left] at (-5,6/2*1.7320508076) {\footnotesize $b_2$};
				\node[left] at (-4.60,3.4*1.7320508076) {\footnotesize.};
				\node[left] at (-4.55,3.45*1.7320508076) {\footnotesize.};
				\node[left] at (-4.50,3.5*1.7320508076) {\footnotesize.};
				\node[left] at (-4.25, 3.75*1.7320508076) {\footnotesize $b_j$};
			\end{scope}
			
			\draw[thick, decorate,decoration={brace, mirror,amplitude=8pt}] (0,0) -- (8,0) node[below] at (4,-0.25) {$2^i$};
			\draw[] (0,0) -- (4,0) -- (2, 2*1.7320508076) -- cycle;
			\draw[] (4,0) -- (8,0) -- (6, 2*1.7320508076) -- cycle;
			\draw[] (2, 2*1.7320508076) -- (6, 2*1.7320508076) -- (5, 3*1.7320508076) -- (3, 3*1.7320508076) -- cycle;
			\draw[dotted] (3, 3*1.7320508076) -- (4, 4*1.7320508076) -- (5, 3*1.7320508076) -- cycle;
			\draw[] (3, 3*1.7320508076) -- (4, 2*1.7320508076) -- (5, 3*1.7320508076) -- cycle;
			\draw[] (3.75, 3.75*1.7320508076) -- (4.25, 3.75*1.7320508076) -- (4, 4*1.7320508076) -- cycle;
		\end{tikzpicture}
		\caption{The vertices used in the calculations of the normal derivative of $f_i$.}
		\label{fig:odometer-construction1}
	\end{figure}
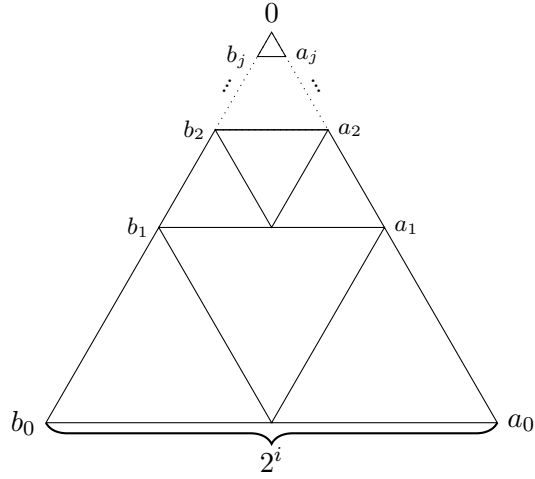
	Using once again the generalized $\frac{1}{5}-\frac{2}{5}$ rule, we obtain 
	\begin{align*}
		\begin{pmatrix}
			f(a_j) \\
			f(b_j)
		\end{pmatrix} 
		= \frac{1}{2} \begin{pmatrix}
			(\frac{3}{5})^j+(\frac{1}{5})^j & (\frac{3}{5})^j-(\frac{1}{5})^j \\
			(\frac{3}{5})^j-(\frac{1}{5})^j & (\frac{3}{5})^j+(\frac{1}{5})^j
		\end{pmatrix} \begin{pmatrix}
			f(a_0) \\
			f(b_0)
		\end{pmatrix} - {\frac{4}{3}} 5^{i-j} (3^j-1),
	\end{align*}
	which together with $f_i(a_0) = -{4/3\cdot} 5^i$ and $f_i(b_0) = 0$ gives 
	\begin{align*}
		&f(a_j) = -2 \Big(\frac{3}{5}\Big)^{j-i} 3^{i} + \frac{2}{3} \Big(\frac{1}{5}\Big)^j, &
		f(b_j) = -2 \Big(\frac{3}{5}\Big)^{j-i} 3^{i} + 2 \Big(\frac{1}{5}\Big)^j,
	\end{align*}
	which can be used to compute the normal derivative $\partial_n$ as below:
	\begin{align*}
		\partial_n f_i(x) = \lim_{j\rightarrow\infty} \Big(\frac{5}{3}\Big)^j \sum_{y\sim_j x} (f_i(x)-f_i(y)) = \lim_{j\rightarrow\infty} (4 \cdot 3^{i}- 8 \cdot 3^{-j-1}) = 4 \cdot 3^{i}.
	\end{align*}
	\textbf{Construction of $\eta$.}
	We construct the function $\eta:B(0,2N)\rightarrow\R$ with $\Delta \eta = \mathds{1}_{B(0,r)}$. For this we take $\eta|_{B(0,r)} = f_l + c$ for some constant $c\in\R$ to be determined. We extend $\eta$ to the ball $B(0,2N)$, such that it is harmonic on the annulus $B(0,2N)\backslash B(0,r)$. For this, we glue a harmonic function $h:B(0,2N)\backslash B(0,r)\rightarrow\R$ and $f_l$ at $\partial B(0,r)$, such that the resulting function $\eta$ is in $\Dom_{L^2}(\Delta)$. Following \cite[Section 5.1]{strichartz-differential-equations-fractals} this results in the two necessary and sufficient conditions
	\begin{align*}
		f_l|_{\partial B(0,r)} + c = h|_{\partial B(0,r)} \quad \text{and} \quad (\partial_n f)|_{\partial B(0,r)} = - (\partial_n h)|_{\partial B(0,r)}.
	\end{align*}
	So it remains to compute $(\partial_n h)|_{\partial B(0,r)}$ which depends on $h|_{\partial B(0,r)} = c$ and to set the constant $c$ accordingly. For the values of $h$ in  $u_i, v_i$ for $i=0,\ldots,L-l$ (see  Figure \ref{fig:odometer-construction2}) we have:
	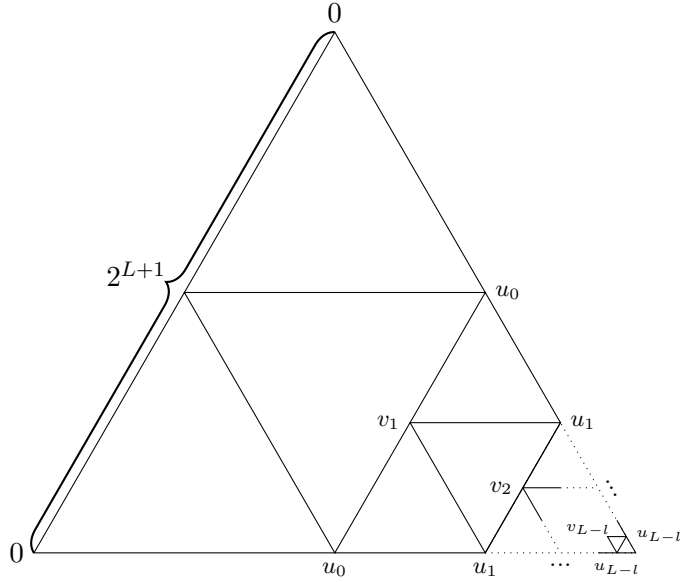
\begin{figure}
		\centering
		\begin{tikzpicture}
			\node[left] at (0,0) {$0$};
			\node[above] at (4,4*1.7320508076) {$0$};
			\node[right] at (6,2*1.7320508076) {\footnotesize $u_0$};
			\node[right] at (7,1*1.7320508076) {\footnotesize $u_1$};
			\node[right] at (7.45,0.55*1.7320508076) {\footnotesize.};
			\node[right] at (7.5,0.5*1.7320508076) {\footnotesize.};
			\node[right] at (7.55,0.45*1.7320508076) {\footnotesize.};
			
			\node[below] at (4,0) {\footnotesize $u_0$};
			\node[below] at (6,0) {\footnotesize $u_1$};
			\node[below] at (7.1,0*1.7320508076) {\footnotesize.};
			\node[below] at (7,0*1.7320508076) {\footnotesize.};
			\node[below] at (6.9,0*1.7320508076) {\footnotesize.};
			
			\node[left] at (5,1*1.7320508076) {\footnotesize $v_1$};
			\node[left] at (6.5,0.5*1.7320508076) {\footnotesize $v_2$};
			
			\draw[thick, decorate,decoration={brace,amplitude=8pt}] (0,0) -- (4,4*1.7320508076) node[left] at (2-0.1,2*1.7320508076+0.25) {$2^{L+1}$};
			\draw[] (0,0) -- (4,0) -- (2, 2*1.7320508076) -- cycle;
			\draw[] (4,0) -- (6,0) -- (7,1*1.7320508076) -- (6, 2*1.7320508076) -- cycle;
			\draw[] (2, 2*1.7320508076) -- (6, 2*1.7320508076) -- (4, 4*1.7320508076) -- cycle;
			\draw[] (5, 1*1.7320508076) -- (6, 0*1.7320508076) -- (7, 1*1.7320508076) -- cycle;
			
			\begin{scope}[shift={(0,-0)}]
				\draw[dotted] (6,0) -- (7,1*1.7320508076) -- (8, 0) -- cycle;
				
				\draw[] (6.5, 0.5*1.7320508076) -- (6.75, 0.25*1.7320508076) ;
				\draw[dotted] (6.75, 0.25*1.7320508076) -- (7, 0*1.7320508076);
				\draw[] (6.5, 0.5*1.7320508076) -- (7, 0.5*1.7320508076) ;
				\draw[dotted] (7, 0.5*1.7320508076) -- (7.5, 0.5*1.7320508076) ;
				\draw[] (7.625, 0.125*1.7320508076) -- (7.75, 0*1.7320508076) -- (7.875, 0.125*1.7320508076) -- cycle;
				
				\draw[] (7.5,0) -- (8,0) -- (7.75,0.25*1.7320508076);
				
				\node[right] at (7.875, 0.125*1.7320508076) {\tiny $u_{L-l}$};
				\node[below] at (7.75, 0*1.7320508076) {\tiny $u_{L-l}$};
				\node[left] at (7.8,0.175*1.7320508076) {\tiny $v_{L-l}$};
			\end{scope} 
		\end{tikzpicture}
		\caption{Vertices used in the calculation of the normal derivative of the  function $h$.}
		\label{fig:odometer-construction2}
	\end{figure}
	\begin{align*}
		h(u_j) = h(u_{L-l}) \frac{1-\left(\frac{3}{5}\right)^{j+1}}{1-\left(\frac{3}{5}\right)^{L-l+1}} \quad \text{and} \quad 
		h(v_j) = h(u_{L-l}) \frac{1-\frac{4}{5}\left(\frac{3}{5}\right)^{j}}{1-\left(\frac{3}{5}\right)^{L-l+1}}.
	\end{align*}
	The normal derivative of a harmonic function on a triangle of size $2^l$ is given by
	\begin{align*}
		\partial_n h(u_{L-l}) = \Big(\frac{3}{5}\Big)^l \Big(2 h(u_{L-l}) - h(u_{L-l+1}) - h(v_{L-l}) \Big) = h(u_{L-l}) \Big(\frac{3}{5}\Big)^l \frac{1}{\left(\frac{5}{3}\right)^{L-l+1}-1}.
	\end{align*}
	By the gluing condition we have $\partial_n h(u_{L-l}) = - \partial_n f_l(u_{L-l})$ and we choose
	\begin{align*}
		c = 4 \cdot 5^l \Big(\Big(\frac{5}{3}\Big)^{L-l+1}-1\Big).
	\end{align*}
	We define  $\eta \in \Dom_{L^2}(\Delta)$ as following:
	\begin{align*}
		\eta(x) := \begin{cases}
			f_l(x) + c, \text{ if } x\in B(0,r) \\
			h(x), \text{ if }  x\in B(0,2N)\backslash B(0,r).
		\end{cases}
	\end{align*}
	\textbf{Calculating the Odometer.}
	Define the obstacle $\gamma$ and its superharmonic majorant $s$ as 
	\begin{align*}
		&\gamma (x) = 3^{L-l} \eta(x) - f_{L+1}(x), &s(x) = \begin{cases}
			\gamma(x), &\text{ if } x\in B(0,2N)\backslash B(0,N) \\
			8/3\cdot 5^L, &\text{ if } x\in B(0,N).
		\end{cases}
	\end{align*}
	Notice that $s$ is constant on $B(0,N)$ taking the value ${8/3}\cdot 5^L$, which is the value of $\gamma$ on $\partial B(0,N)$. Clearly $s$ is superharmonic on $B(0,2N)$, since $\Delta s = \mathds{1}_{B(0,2N)\backslash B(0,N)}$. It remains to show that $s\geq \gamma$. For $x\in B(0,2N)\backslash B(0,N)$ this is obvious by definition, so it is sufficient to show $ {8/3}\cdot 5^L \geq \gamma(x)$ for all $x\in B(0,N)$. Since $\gamma$ is subharmonic on $B(0,r)$ it attains its maximum on the boundary and
	$$\gamma|_{\partial B(0,r)}= \frac{4}{3} 5^l (3^{L-l+1}-1) < \frac{8}{3} \cdot 5^L,$$ 
	it suffices to investigate the values of $\gamma$ in the annulus $B(0,N)\backslash B(0,r)$.
	The following result states that it actually suffices to consider only junction points approaching $\partial B(0,N)$.
\begin{lem}\label{lem:example_maximum_superharmonic}
Let $T^n=T^n(A,B,C)\subset \SG_n$, for some $n\in \N$ be the discrete triangle
with boundary points $A,B,C$ and $a,b,c$ be the midpoints opposing $A,B,C$, respectively. Furthermore let $f:T^n\rightarrow\R$ be a function with $\Delta_n f = -1$ and $f(A) \geq f(B) \geq f(C)$. If 
			$$f(A)\geq \max\{f(a),f(b),f(c)\} \text{ and } f(A)=\max\{f(x):\ x\in T^n(A,b,c)\},$$
			then $ \max_{x\in T(A,B,C)} f(x) = f(A)$.
		\end{lem}
		\begin{proof}
Let $h$ be the unique harmonic function on $T^n$ taking boundary values $h|_{\partial T^n}= f|_{\partial T^n}$. Since $\Delta_n f = -1$, we have the decomposition
			$$ f(x) = h(x) + \sum_{y\in T^n} \greenFunctionDiscrete^n_{T^n}(x,y),$$
			where $ \greenFunctionDiscrete^n_{T^n}$ is the discrete Green function defined in \eqref{eq:green-visits}. Notice that $ s' = \sum_{y\in T^n} g_{T^n}^n(x,y)$ is rotationally symmetric, i.e.~$s'$ takes the same values on subtriangles $T^n(A,b,c),T^n(a,B,c)$, and $T^n(a,b,C)$.
			Since $h(A)\geq h(B) \geq h(C)$, after appropriate rotations $\varphi_B,\varphi_C$ we get $$ h|_{T^n(A,b,c)} \geq h|_{T^n(a,B,c)} \circ \varphi_B \geq h|_{T^n(a,b,C)} \circ \varphi_C, $$
			from which we can conclude $ f|_{T^n(A,b,c)} \geq f|_{T^n(a,B,c)} \circ \varphi_B \geq f|_{T^n(a,b,C)} \circ \varphi_C$.
		\end{proof}
	We finally  calculate $\gamma$ for all  $x_k^j,y_k^j,z_k^j$, for $k\in \N$ and $j=1,\ldots,L-l-1$, as in Figure \ref{fig:odometer-construction3}.
	\begin{figure}
		\centering
		\resizebox{\textwidth}{!}{\begin{tikzpicture}[scale=0.9]
				\node[right] at (6,2*1.7320508076) {\footnotesize $x_0^0$};
				\node[right] at (7,1*1.7320508076) {\footnotesize $x_0^1$};
				\node[right] at (7.45,0.55*1.7320508076) {\footnotesize.};
				\node[right] at (7.5,0.5*1.7320508076) {\footnotesize.};
				\node[right] at (7.55,0.45*1.7320508076) {\footnotesize.};
				
				\node[below] at (4,0) {\footnotesize $x_0^0$};
				\node[below] at (6,0) {\footnotesize $x_0^1$};
				\node[below] at (7.1,0*1.7320508076) {\footnotesize.};
				\node[below] at (7,0*1.7320508076) {\footnotesize.};
				\node[below] at (6.9,0*1.7320508076) {\footnotesize.};
				
				\node[left] at (5,1*1.7320508076) {\footnotesize $y_0^1$};
				\node[left] at (6.5,0.5*1.7320508076) {\footnotesize $y_0^2$};
				
				\draw[thick, decorate,decoration={brace,amplitude=8pt}] (0,0) -- (4,4*1.7320508076) node[left] at (2-0.1,2*1.7320508076+0.25) {$2^{L+1}$};
				\draw[] (0,0) -- (4,0) -- (2, 2*1.7320508076) -- cycle;
				\draw[] (4,0) -- (6,0) -- (7,1*1.7320508076) -- (6, 2*1.7320508076) -- cycle;
				\draw[] (2, 2*1.7320508076) -- (6, 2*1.7320508076) -- (4, 4*1.7320508076) -- cycle;
				\draw[] (5, 1*1.7320508076) -- (6, 0*1.7320508076) -- (7, 1*1.7320508076) -- cycle;
				\draw[] (7.625, 0.125*1.7320508076) -- (7.75, 0*1.7320508076) -- (7.875, 0.125*1.7320508076) -- cycle;
				
				\draw[dotted] (6,0) -- (7,1*1.7320508076) -- (8, 0) -- cycle;
				\draw[] (6.5, 0.5*1.7320508076) -- (6.75, 0.25*1.7320508076) ;
				\draw[dotted] (6.75, 0.25*1.7320508076) -- (7, 0*1.7320508076);
				\draw[] (6.5, 0.5*1.7320508076) -- (7, 0.5*1.7320508076) ;
				\draw[dotted] (7, 0.5*1.7320508076) -- (7.5, 0.5*1.7320508076) ;
				\draw[] (7.625, 0.125*1.7320508076) -- (7.75, 0*1.7320508076) -- (7.875, 0.125*1.7320508076) -- cycle;
				\draw[] (7.5,0) -- (8,0) -- (7.75,0.25*1.7320508076);
				
				\begin{scope}[shift={(9,1)},scale=0.8]
					\node[right] at (8,0) {\footnotesize $x_0^j$};
					\node[left] at (0,0) {\footnotesize $y_0^j$};
					\node[above] at (4,4*1.7320508076) {\footnotesize $x_0^{j-1}$};
					\node[right] at (6,2*1.7320508076) {\footnotesize $x_1^j$};
					\node[right] at (5,6/2*1.7320508076) {\footnotesize $x_2^j$};
					\node[right] at (4.60,3.4*1.7320508076) {\footnotesize.};
					\node[right] at (4.55,3.45*1.7320508076) {\footnotesize.};
					\node[right] at (4.50,3.5*1.7320508076) {\footnotesize.};
					\node[right] at (4.125, 3.875*1.7320508076) {\footnotesize $x_k^j$};
					
					\begin{scope}[shift={(8,0)}]
						\node[left] at (-6,2*1.7320508076) {\footnotesize $y_1^j$};
						\node[left] at (-5,6/2*1.7320508076) {\footnotesize $y_2^j$};
						\node[left] at (-4.60,3.4*1.7320508076) {\footnotesize.};
						\node[left] at (-4.55,3.45*1.7320508076) {\footnotesize.};
						\node[left] at (-4.50,3.5*1.7320508076) {\footnotesize.};
						\node[left] at (-4.125, 3.875*1.7320508076) {\footnotesize $y_k^j$};
					\end{scope}
					
					\node[below] at (4,2*1.7320508076) {\footnotesize $z_1^j$};
					\node[below] at (4,3*1.7320508076) {\footnotesize $z_2^j$};
					\node[below] at (4.05,3.875*1.7320508076) {\footnotesize $z_k^j$};
					
					\draw[thick, decorate,decoration={brace, mirror,amplitude=8pt}] (0,0) -- (8,0) node[below] at (4,-0.25) {$2^{L-1-j}$};
					
					\draw[] (3.75, 3.25*1.7320508076) -- (4, 3*1.7320508076) -- (4.25, 3.25*1.7320508076);
					\draw[dotted] (3.5, 3.5*1.7320508076) -- (3.75, 3.25*1.7320508076);
					\draw[dotted] (4.25, 3.25*1.7320508076) -- (4.5, 3.5*1.7320508076);
					\draw[] (3.875, 3.875*1.7320508076) -- (4, 3.75*1.7320508076) -- (4.125, 3.875*1.7320508076) -- cycle;
					\draw[] (0,0) -- (4,0) -- (2, 2*1.7320508076) -- cycle;
					\draw[] (4,0) -- (8,0) -- (6, 2*1.7320508076) -- cycle;
					\draw[] (2, 2*1.7320508076) -- (6, 2*1.7320508076) -- (5, 3*1.7320508076) -- (3, 3*1.7320508076) -- cycle;
					\draw[dotted] (3, 3*1.7320508076) -- (4, 4*1.7320508076) -- (5, 3*1.7320508076) -- cycle;
					\draw[] (3, 3*1.7320508076) -- (4, 2*1.7320508076) -- (5, 3*1.7320508076) -- cycle;
					\draw[] (3.75, 3.75*1.7320508076) -- (4, 4*1.7320508076) -- (4.25, 3.75*1.7320508076);
					
				\end{scope}
				\fill[opacity=0.25] (7.625, 0.125*1.7320508076) -- (7.75, 0.25*1.7320508076) -- (7.875, 0.125*1.7320508076) -- cycle;
				\draw [->] (7.75, 0.175*1.7320508076) to [out=-50,in=230] (12,0);
		\end{tikzpicture}}
		\caption{Points needed to show that $ {8/3} \cdot 5^L \geq \gamma(x)$ for all  $x\in B(0,N)\backslash B(0,n)$.}
		\label{fig:odometer-construction3}
	\end{figure}
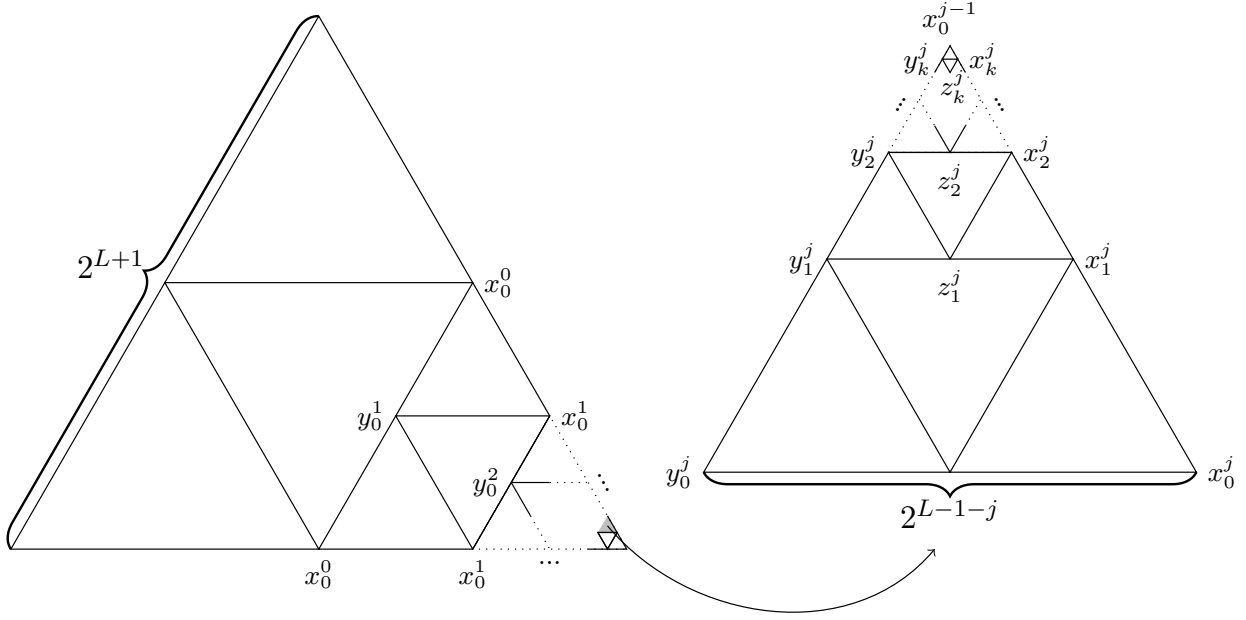
	\begin{align*}
		&\gamma(x_0^j) = \frac{4}{3} 5^{L-j} (3^{j+1}-1) \leq \gamma(x_0^{j-1}), &\gamma(y_0^j) = \frac{4}{3} 5^{L-j} (4\cdot 3^{j}-2) \leq \gamma(x_0^{j-1}) 
	\end{align*}
	and 
	\begin{align*}
		\gamma(x_k^j) &= \gamma(x_0^{j-1}) - \frac{2}{3} (3^{j+1}-9) \Big(\frac{3}{5} \Big)^k 5^{L-j} - \frac{2}{3}(3^j+1) \Big(\frac{1}{5} \Big)^k 5^{L-j}, \\
		\gamma(y_k^j) &= \gamma(x_0^{j-1}) - \frac{2}{3} (3^{j+1}-9) \Big(\frac{3}{5} \Big)^k 5^{L-j} + \frac{2}{3}(3^j-3) \Big(\frac{1}{5} \Big)^k 5^{L-j}, \\
		\gamma(z_k^j) &= \gamma(x_0^{j-1}) - \frac{8}{15} (3^{j+1}-9) \Big(\frac{3}{5} \Big)^k 5^{L-j} + \frac{8}{15} \Big(\frac{1}{5} \Big)^k 5^{L-j},
	\end{align*}
	which together with Lemma \ref{lem:example_maximum_superharmonic} implies that on all triangles $T^n(x_0^{j-1},x_0^j,y_0^j)\subseteq \mathsf{SG}_n$ the obstacle $\gamma$ attains its maximum in $x_0^{j-1}$ for all $n\in \N$. Since $\gamma(x_0^j)\leq \gamma(x_0^{j-1})$, $\gamma$ attains its maximum in $\big(B(0,N)\backslash B(0,r)\big)\cap \mathsf{SG}_n$ at $x_0^0$ for any $n\in\N$. Because $\gamma$ is continuous we get
	\begin{align*}
		\frac{8}{3} \cdot 5^L = \gamma(x_0^0) \geq \gamma(x) \text{ for all } x\in B(0,N)\backslash B(0,r).
	\end{align*}
	
	\section{Boundary Regularity}\label{sec:appendixC}
	
This final part focuses on properties of the boundary of the noncoincidence set $D\subset\SG$  obtained as the solution of the obstacle problem on the Sierpi\'nski gasket $\SG\subset \R^2$. We will use a generalization of Lebesgue's density theorem for $\SG$.
	
	We first consider the case when the density function $\sigma:\SG\to[0,\infty)$ is  continuous. Recall that by Lemma \ref{theo:convergenceSandpileOdometer} for $\sigma$, the odometer function of the obstacle problem is the limit of odometer functions $u_n$ for the divisible sandpile on $\mathsf{SG}_n$ started from initial densities $\sigma_n$ 
		that fulfill  \eqref{cond:RROdometerConvergence1} and \eqref{cond:RRconditionOdometer1}-\eqref{cond:RRconditionOdometer3} as well as the standard assumptions on initial configurations (\ref{cond:Starting1})-(\ref{cond:Starting4}). A suitable and canonical choice would be to choose $\sigma_n(y)$ as the average value of $\sigma$ in a double triangle around $y$ as in Theorem \ref{thm:uniform_scaling_limit}. We will always fix $\sigma_n$ to be defined this way in the following proofs. 
	\begin{prop}\label{prop:bndr-reg-cont}
		Let $\sigma:\SG\to[0,\infty)$ be continuous with compact support such that $\mu(\sigma^{-1}(1))=0$ and $D=\{u>0\}$, where $u$ is the odometer function of the obstacle problem for $\sigma$. Then
		\vspace{-0.25cm}
		\begin{enumerate}
			\setlength\itemsep{0cm}
			\item $\mu(\partial D)=0$,
			\item $\int_D \sigma d\mu=\mu(D)$.
		\end{enumerate}
	\end{prop}
	\begin{proof}
		\textit{(1)}. Fix $\lambda$ with $0<\lambda<1$, and let $x\in\partial D$ with $\sigma(x)\leq \lambda$. We can then find $\varepsilon>0$ such that $\sigma$ is less than $(1+\lambda)/2$ on $B(x,\varepsilon).$ 
		We choose $\sigma_n$ on $\mathsf{SG}_n$ converging to $\sigma$ as discussed above and denote by $u_n$ the odometer function of the divisible sandpile started from $\sigma_n$ and by $\DD_n$ the corresponding divisible sandpile cluster. By Theorem \ref{theo:convergenceSandpileOdometer}, $u_n$ converges to $u$ uniformly, so for $n$ large enough,  $\sigma_n(y)\leq (1+\lambda)/2$ for all $y\in B(x,\varepsilon)\cap \mathsf{SG}_n.$ Next we choose a point $x_n'\in\partial \DD_n$ that is at a distance of less than $\varepsilon/3$ from $x$ (which we can do for large enough $n$, since we know that the sets $\DD_n$ converge to $D$) and we then choose  $x_n\in \SG_n$ of distance less than $\varepsilon/3$ from $x_n'$, such that $x_n\in \DD_n$ (and hence $u_n(x_n)>0$). Define the function
		\begin{align*}
			w_n(y)=u_n(y)+\frac{1}{2}(1-\lambda)\delta_n^{\beta}\mathbb{E}_y\big[\tau^n_{\delta_n^{-1}\varepsilon/3}(x_n)\big],
		\end{align*}
		where  as before $\mathbb{E}_y\big[\tau^n_{\delta_n^{-1}\varepsilon/3}(x_n)\big]$
		denotes the expected exit time from the ball $B^n(x_n,\delta_n^{-1}\varepsilon/3)$ of a simple random walk on $\SG_n$ started at $y\in \SG_n$. Since 
		$w_n$ is subharmonic on $\DD_n\cap B^n(x_n,\delta_n^{-1}\varepsilon/3)$, it attains its maximum on the boundary, but for $y\in\partial \DD_n$, we also have $w_n(y)\leq w_n(x_n)$,
		hence the maximum is attained in $\partial B^n(x_n,\delta_n^{-1}\varepsilon/3)$. Thus there exists $y_n\in\partial B^n(x_n,\delta_n^{-1}\varepsilon/3)$ with
		\begin{align*}
			u_n(y_n)\geq \frac{1}{2}C(1-\lambda)\varepsilon^\beta,
		\end{align*}
		for some constant $C>0$. We remark that the Laplacian of $u_n$ is bounded by some constant $M>0$ and $u_n=0$ on  $\partial \DD_n$, hence  similar to the proof of Lemma \ref{lem:growth_odometer_rotor} we conclude
		\begin{align*}
			u_n(y_n)\leq cMd(y_n,\partial \DD_n)^\beta\delta_n^\beta, \text{ for some } c>0.
		\end{align*}
By the two previous inequalities and with $C':=\big(C\frac{1-\lambda}{2cM}\big)^{1/\beta}$ we obtain $d(y_n,\partial \DD_n)\geq C'\varepsilon\delta_n^{-1}$.    
	Thus $B^n(y_n,C'\varepsilon\delta_n^{-1})\subseteq \DD_n$, and for any $x\in\partial D\cap\{\sigma\leq\lambda\}$ and suitable $C''>0$ we get
		\begin{align*}
			\mu(B(x,(1+C')\varepsilon)\cap (\partial D)^C)\geq C''\varepsilon^\alpha.
		\end{align*}
		By the Lebesgue's density theorem \cite{federer2014geometric} for the Hausdorff measure on the gasket  we have
		\begin{align*}
			\mu(\partial D\cap \{\sigma\leq \lambda\})=0,
		\end{align*}
		and by letting $\lambda\rightarrow 1$ and using the fact that $\mu(\{\sigma=1\})=0$ we obtain the first claim.
		
		\textit{(2).} For the sandpile clusters $\DD_n\subset\SG_n$ with initial density $\sigma_n$ converging to $\sigma$ we have
		\begin{align*}
			|\DD_n|\leq\sum_{z\in \DD_n\cup \partial \DD_n}\sigma_n(z)\leq|\DD_n|+|\partial\DD_n|.
		\end{align*}
		For given $\varepsilon>0$ and $n$ large enough it holds
		\begin{align*}
			D_\varepsilon \cap \mathsf{SG}_n \subseteq \DD_n \subseteq \DD_n\cup \partial \DD_n \subseteq D^\varepsilon,
		\end{align*}
		hence
		\begin{align*}
			\int_{D_\varepsilon}\sigma d\mu\leq \liminf_{n\rightarrow\infty}\int_{D_\varepsilon}\sigma^\doubleTriangle_n d\mu\leq \mu((D_n\cup\partial D_n)^\doubleTriangle)\leq \mu(D^\varepsilon),
		\end{align*}
		and similarly
		\begin{align*}
			\int_{D^\varepsilon}\sigma d\mu\geq \mu(D_\varepsilon).
		\end{align*}
		Putting together the previous two inequalities we obtain
		\begin{align*}
			\int_{D}\sigma d\mu = \int_{D_\varepsilon} \sigma d\mu+\int_{D\backslash D_\varepsilon}\sigma d\mu\leq \mu(D)+\mu(D^\varepsilon\backslash D)+M\mu(D\backslash D^\varepsilon),
		\end{align*}
		and
		\begin{align*}
			\int_{D}\sigma d\mu \geq \mu(D)-\mu(D\backslash D^\varepsilon)-M\mu(D^\varepsilon\backslash D),
		\end{align*}
		where $M>0$ is as in \eqref{cond:Starting1}. Using $\mu(\partial D)=0$ and letting $\varepsilon\to 0$ gives the claim.
	\end{proof}
	We also consider densities $\sigma$ that are not necessarily continuous everywhere and we investigate the corresponding noncoincidence set $D$.
	
	\begin{prop}\label{prop:boundary_conservation_of_mass}
		Let $\sigma:\SG\to[0,\infty)$ be continuous almost everywhere and assume that there exists $\lambda>0$ such that for all $x\in \mathsf{SG}$ we have either $\sigma(x)\leq \lambda$ or $\sigma(x)\geq 1.$ If $D=\{u>0\}$ is the solution to the obstacle problem for $\sigma$ and $\widetilde{D}=D\cup\{\sigma\geq 1\}^\circ$ we then have
		\vspace{-0.25cm}
		\begin{enumerate}
			\item $\mu(\partial \widetilde{D})=0$,
			\item $\mu(\widetilde{D})=\int_{\widetilde{D}}\sigma d\mu$.
		\end{enumerate}
	\end{prop}
	\begin{proof}
		\textit{(1).} We choose $x\in\partial \widetilde{D}$. Since the discontinuity points of $\sigma$ are a null set, we can assume that $\sigma$ is continuous at $x$. If $\sigma(x)\leq\lambda$, we can choose $\varepsilon>0$ small enough such that $\sigma\leq (1+\lambda)/2$ in $B(x,\varepsilon)$. If we define the discrete particle configurations $\sigma_n$ as in  \eqref{eq:sigma-n} by taking the average of $\sigma$ around a given point $y$, then $\sigma_n$ also eventually are less than $(1+\lambda)/2$ in $B(x,\varepsilon)$, and from here we  can then proceed as in the proof of  Proposition \ref{prop:bndr-reg-cont}\textit{(1)}. 
		If $\sigma(x)\geq 1$, since $\sigma$ is continuous at $x$ we can find an $\varepsilon > 0$ such that $\sigma\geq 1$ on $B(x,\varepsilon)$, but this would imply that $x\in\{\sigma\geq 1\}^\circ$, hence $x$ cannot be a point on the boundary, thus completing the proof of the first claim. The second claim works as in the proof of Proposition \ref{prop:bndr-reg-cont}\textit{(2)}. 
	\end{proof}
\end{appendices}
\textbf{Comments and other direction of work.} An interesting question to consider is the scaling limit of the Abelian sandpile model, where instead of a continuous amount of mass we only allow whole particles to be distributed among the vertices. In \cite{pegden-smart-abelian-scaling}, the authors prove that on $\Z^d$, when starting with $n$ particles at the origin and stabilizing until each site has less than $ 2d$ particles, as $n\to\infty$, the set of occupied sites
admits a weak* scaling limit on $\R^d$. The method can by no means be extended to fractal objects, and a completely different approach that makes use of the recursive structure of the gasket  should be used in order to prove that also the scaling limit for the Abelian sandpile model with multiple sources is the same as for the other three aggregation models.
The method used in the current work can be extended to other nested fractals, but due to the technicalities already appearing on the gasket, we have chosen not to work on the most general case of such self-similar sets.


\textbf{Acknowledgments.} The research of Robin Kaiser and Ecaterina Sava-Huss is supported by the Austrian Science Fund (FWF) P34129.  We are very grateful to the referee for a very careful reading of the paper and for the many valuable comments and suggestions, which substantially improved the paper.
\bibliography{idla}
\bibliographystyle{alpha}

\textsc{Uta Freiberg}, Department of Mathematics, Chemnitz University of Technology, Germany\\
\texttt{Uta.Freiberg@mathematik.tu-chemnitz.de}

\textsc{Nico Heizmann}, Department of Mathematics, Chemnitz University of Technology, Germany\\
\texttt{nico.heizmann@math.tu-chemnitz.de}

\textsc{Robin Kaiser}, Institut für Mathematik, Universität Innsbruck, Austria.\\
\texttt{Robin.Kaiser@uibk.ac.at}

\textsc{Ecaterina Sava-Huss}, Institut für Mathematik, Universität Innsbruck, Austria.\\
\texttt{Ecaterina.Sava-Huss@uibk.ac.at}

\end{document}